%% file: testing.tex
\documentclass[11pt]{article}
\usepackage[utf8]{inputenc}
\usepackage{amsmath,amsthm}
\usepackage{amsfonts}
\usepackage{amssymb}
\usepackage{array}
\usepackage{bm}
\usepackage{multirow}
\usepackage{subfig}
\usepackage{enumitem}
\usepackage[normalem]{ulem}
\usepackage[font=small]{caption}
\usepackage{graphicx}
\usepackage{setspace}
\usepackage{xspace}
\usepackage{mathtools}
\usepackage{multirow}
\usepackage{makecell}

\usepackage{tikz} 
\usetikzlibrary{arrows,matrix,positioning,fit,calc}

\usepackage[margin=2.9cm]{geometry}
\RequirePackage[colorlinks,citecolor=blue,urlcolor=blue]{hyperref}

\usepackage{algorithm}
\usepackage{algpseudocode}

\algnewcommand{\IIf}[1]{\State\algorithmicif\ #1\ \algorithmicthen}

\usepackage[
		style=alphabetic,
		backend=bibtex,
		doi=false,
		isbn=false,
		url=false,
		firstinits=true,
		maxnames=4
		]{biblatex}

\bibliography{refs}

\makeatletter
\def\blx@maxline{77}
\makeatother

\renewbibmacro{in:}{}           
\DeclareFieldFormat[article,inbook,incollection,inproceedings,patent,thesis,unpublished]{citetitle}{#1}
\DeclareFieldFormat[article,inbook,incollection,inproceedings,patent,thesis,unpublished]{title}{#1} 

\usepackage{authblk}
\title{Non-Asymptotic Chernoff Lower Bound and Its Application to Community Detection in Stochastic Block Model}
\author{Zhixin Zhou and Ping Li}
\date{}

\input{bi_pl_defs}

\newcommand{\reals}{\mathbb R}
\newcommand{\infnorm}[1]{\|{#1}\|_\infty}

\begin{document}
\maketitle

\begin{abstract}
	Chernoff coefficient is an upper bound of Bayes error probability in classification problem. In this paper, we will develop sharp Chernoff type bound on Bayes error probability. The new bound is not only an upper bound but also a lower bound of Bayes error probability up to a constant in a non-asymptotic setting. Moreover, we will apply this result to community detection in stochastic block model. As a clustering problem, the optimal error rate of community detection can be characterized by our Chernoff type bound. This can be formalized by deriving a minimax error rate over certain class of parameter space, then achieving such error rate by a feasible algorithm employ multiple steps of EM type updates.
\end{abstract}

\section{Introduction}

Many classification and clustering problems can reduce to a \emph{symmetric hypothesis testing} problem.  In a classical setting, given two hypothesis $H_0$ and $H_1$, where $H_i$ assumes that observing data from a measurable space with distribution $P_i$, one discriminates between them according to certain decision rule.  Type-I error occurs if one accepts $H_0$ while the data are generated from distribution $P_1$, and vice versa on Type-II error. Symmetric hypothesis testing indicates that the hypotheses are equiprobable, and the loss function weights type-I error and type-II error equally. Hence, we would like to focus on \emph{Bayes error probability}, which average two kinds of error probability.

The asymptotic and non-asymptotic behavior of Bayes error probability becomes an essential problem in symmetric hypothesis testing. Given probability density function (PDF) or probability mass function (PMF) $\varphi_0$ and $\varphi_1$ of distribution $P_0$ and $P_1$ respectively, the \emph{Chernoff information}, defined as
\begin{align}\label{eq:chernoff:inf}
\cds{\varphi_0}{\varphi_1} = -\sup_{\alpha\in(0,1)} \log \int \varphi_0^{1-\alpha}  \varphi_1^\alpha d\mu
\end{align}
plays an important role in the exponent of Bayes error since it was introduce in \cite{chernoff1952measure}. A Chernoff type lower bound was investigated in \cite{shannon1967lower}. A similar lower bound was introduced in \cite{blahut1974hypothesis}. It is still a very powerful tools in recent researches, such as community detection \cite{abbe2015community, gao2018community, zhou2018optimal} and quantum information theory \cite{nussbaum2009chernoff, dalai2013lower}. However, the ratio between \emph{Chernoff coefficient}, defined as $\exp(-\cds{\varphi_0}{\varphi_1})$,  and the Bayes error probability has not been answered in previous literature. In this paper, we will propose a new Chernoff type bound such that its ratio with Bayes error probability is upper and lower bounded by constants. Although a comparable ``second-order" asymptotics for \emph{asymmetric hypothesis testing} was investigated in \cite{li2014second, zhou2018second}, there is no direct application to the symmetric case. Indeed, a non-asymptotic bound in the symmetric case requires extra effort.

This paper will also apply the main result of Chernoff type upper and lower bound to one of popular clustering problems in statistics, namely community detection. Particularly, we will focus on \emph{stochastic block model} (SBM). Many theories and effective algorithms have been proposed for solving in SBMs, including global approaches such as spectral clustering~\cite{ rohe2011spectral, krzakala2013spectral, lei2013consistency, fishkind2013consistent, vu2014simple, massoulie2014community, yun2014accurate, bordenave2015non, gulikers2017spectral, pensky2017spectral, zhou2018spectral} and convex relaxations via  semidefinite programs (SDPs)~\cite{amini2014semidefinite, hajek2016achieving, bandeira2015random, guedon2016community, montanari2016semidefinite, ricci2016performance, agarwal2017multisection, perry2017semidefinite}. Global approaches usually involve a single optimization step (truncated SVD in spectral method and SDP in convex relaxations method) and do not require good initialization. However, these algorithms are usually not optimal on their own, because both SVD and SDP lose the block structure in SBM. The pseudo-likelihood approach~\cite{amini2013pseudo}
filled in the gap with local refinement and makes optimal clustering possible. The general idea was concluded as ``Good Initialization followed by Fast Local Updates'' (GI-FLU) by \cite{gao2017achieving}. Since the minimax error rate proposed in \cite{zhang2016minimax},  algorithms in the manner of GI-FLU are developed in \cite{chin2015stochastic, gao2017achieving, abbe2015community, zhou2018optimal}. However, as the optimal Chernoff upper and lower bounds were not used in these papers, the minimax rate is not sufficiently accurate and very few algorithm have been proved to be optimal. Details can be found in the following table. Here, $n$ is the number of nodes and $d$ is the average degree of a node in the network. $K$ denotes number of communities. $D$ indicates Chernoff information in \eqref{eq:chernoff:inf}, though it might have different notations in other papers. $o(1)$ is some unspecified positive sequences converging to 0.

\begin{table}[h]
	\caption{Comparison with existing results.}
	\begin{tabular}{c|c|c|c|c|c}
		\hline
		paper& density & symm. & $p$ vs $q$ & minimax error &algorithmic error  \\ \hline
		\cite{chin2015stochastic} & not needed & yes & needed & not derived & $\exp(-CD)$. $(C<1)$  \\ \hline
		\cite{abbe2015community} & $\Theta(\log n)$ & no  & not needed & not derived & $o(1/n)$   \\ \hline
		\cite{gao2017achieving} & not needed & yes & needed & $\exp(-(1+o(1))D)$ &  $\exp(-(1-o(1))D)$  \\ \hline
		\cite{zhou2018optimal}& $O(\sqrt n)$ & no & not needed & $\Omega(\exp(-D)/d^{K/2})$ & $O(\exp(-D)/\sqrt d)$\\ \hline
		ours & not needed & yes & not needed & $\Omega(\exp(-D)/\sqrt d)$ & $O(\exp(-D)/\sqrt d)$ \\ \hline
	\end{tabular}
\end{table}
Some features or assumptions of the problem are described as follows. \emph{Density} indicates the average degree of a node. 
\emph{Symmetry (symm.)} means the paper assumes that the network is an undirected graph. Community detection on symmetric network is usually more difficult since half edges are duplicated. If a paper assumes ``$p$ vs $q$", that means the the probability of connections within the same community is higher than the ones between different communities. This setting can be generalized to assortative condition in \cite{amini2014semidefinite}. The algorithm and its analysis are simpler than general SBM without this assumption. \emph{Minimax error rate} can be considered as fundamental limit of community detection problem. \emph{Algorithmic error rate} are the theoretical guarantees of feasible algorithms in different papers..\par

Block partitioning skills introduced in \cite{chin2015stochastic} generate enough independence between different steps of their algorithm. However, the last local update can only apply on half of dataset, so the error rate is much higher than $\exp(-D)$. Algorithm derived in \cite{gao2017achieving} has error rate similar as the minimax error rate in \cite{zhang2016minimax}, but the term $\exp(o(1)D)$ can be arbitrary divergence sequence. The analysis in \cite{abbe2015community} focus on the density regime $\Theta(\log n)$, but it cannot generalize to other densities or symmetric case. To achieve an optimal error rate, the algorithm in \cite{zhou2018optimal} allows twice local update. However, their approach cannot extend to undirected network. We will combine different existing techniques and propose a new algorithm that achieves the minimax error rate (up to a constant).\par
We summarize the contributions of this paper as follows:
\begin{itemize}
	\item[1.] We demonstrate sharp non-asymptotic Chernoff type upper and lower bound for Bayes error probability.
	\item[2.] We proposed a sharp non-asymptotic minimax lower bound for community detection in general SBMs.
	\item[3.] We provide a feasible algorithm which guarantees to achieve the minimax lower bound up to a constant.
\end{itemize}

The rest of the paper will be organized as follows. We introduce the Chernoff type upper and lower bound in Section \ref{sec:chernoff}, then we present our minimax lower bound and the provable community detection algorithm with its analysis in Section \ref{sec:community}. Simulations will appear in Section~\ref{sec:simulation}. Proofs of Theorems  and corollaries in Section~\ref{sec:chernoff} will appear in Section \ref{sec:proofs}. Proofs about minimax error rate and consistency of community detection can be found in Section~\ref{sec:proofs:sec:3}.


\section{Non-asymptotic Chernoff upper and lower bounds}\label{sec:chernoff}

We will introduce a fundamental testing problem under a Bayes setting, then describe its relation with Chernoff coefficient. As part of main contributions of this paper, we will present a new Chernoff type upper and lower bound of Bayes error probability. We will also introduce its application to distribution in exponential family as a useful example.

\subsection{General cases}\label{sec:general:cases}

We will define a symmetric hypothesis testing problem and its Bayes error probability. Let $\varphi_{0j}$ and $\varphi_{1j}$ for $j\in [n]$ be two sequences of PDFs for one-dimensional real random variables. Same results hold if they are PMFs, but we only consider PDFs for brevity. We assume for every $j\in\mathbb [n]$, $\varphi_{0j}$ and  $\varphi_{1j}$ are defined on the same measure space $(\Omega_j, \Sigma_j, \mu)$. Let us write 

\begin{align}\label{eq:def:prod:space}
\begin{split}
\Omega := \prod_{j=1}^n \Omega_j, \quad & \Sigma := \bigotimes_{j=1}^n  \Sigma_j, \quad \text{and}\\
\varphi_z(x) := \varphi_z(x_1,\dots,x_n) &:= \prod_{j=1}^n \varphi_{zj}(x_j)  \quad \text{for } z\in\{0,1\}.
\end{split}
\end{align}
Furthermore, we denote the \emph{Kullback–Leibler divergence} from $\varphi_1$ to $\varphi_0$ as
\begin{align*}
\kld{\varphi_0}{\varphi_1} = \int_\Omega \varphi_0 \log \f{\varphi_1}{\varphi_0} d\mu.
\end{align*}
We assume both $\kld{\varphi_0}{\varphi_1}$ and $\kld{\varphi_1}{\varphi_0}$ are positive real numbers, which implies $\int_\Omega (\varphi_0+\varphi_1) \big|\log \f{\varphi_1}{\varphi_0}\big| d\mu<\infty$. In particular, it requires $\varphi_0$ and $\varphi_1$ have the same support, but take different values on a set with non-zero measure. For a pair of PDFs satisfying these conditions, we say
\begin{align}\label{eq:density:family}
(\varphi_0, \varphi_1)\in \mathcal F (\Omega, \Sigma, \mu, n).
\end{align}
Now we randomly draw a number $z\in\{0,1\}$ with equal probability $1/2$, and draw a random sample $X = \{X_1,\dots,X_n\}$ where $X_j\sim \varphi_{zj}$ independently by definition. We are interested in recovering $z$ given $X=x$. For any estimator $\hat z := \hat z(x)$ of $z$, we define the \emph{Bayes error probability} and the corresponding \emph{Bayes estimator} as
\begin{align*}
R(\hat z, z) := \f 12 \sum_{z\in\{0,1\}}\P(\hat z\ne z)
\quad\text{where}\quad
\hat z := \arg \max_{z \in \{0,1\}} \varphi_z(x).
\end{align*}
Bayes estimator is the best estimator by Neyman-Pearson lemma. The Bayes error probability is closely related to \emph{total variation affinity} between $\varphi_0$ and $\varphi_1$, denoted as $\eta(\varphi_0, \varphi_1)$, which will be defined as follows:
\begin{align}\label{eq:tv:affinity}
\begin{split}
\eta(\varphi_0, \varphi_1) &:= \int_\Omega \min(\varphi_0, \varphi_1) d\mu 
= \int_\Omega \varphi_0 1\{\varphi_0 \le \varphi_1\} d \mu  + \int_\Omega \varphi_1 1\{\varphi_1 < \varphi_0\} d \mu = 2R(\hat z, z).
\end{split}
\end{align}
The naming of total variation affinity comes from the fact that
\begin{align*}
\eta(\varphi_0, \varphi_1) = 1- \tv{\varphi_0}{\varphi_1},
\quad\text{where}\quad
\tv{\varphi_0}{\varphi_1} := \sup_{A\in\Sigma} \Big|\int_A \varphi_0-\varphi_1 d\mu\Big|.
\end{align*}
Now we can focus on the total variation affinity and express it as
\begin{align}
\label{eq:bayes:min}
\eta(\varphi_0, \varphi_1)
= \int_\Omega \min(\varphi_0,\varphi_1) d\mu
= \int_\Omega \varphi_0^{1-\alpha} \,\varphi_1^{\alpha} \, \min(\lr^{\alpha},\lr^{\alpha-1}) d\mu.
\end{align}
where $\lr= \varphi_0/\varphi_1$ is the likelihood ratio defined point-wisely on $\Omega$. We observe that $\varphi_0^{1-\alpha} \,\varphi_1^{\alpha}$ is a PDF on $\Omega$ up to a normalizer and $\min(\lr^{\alpha},\lr^{\alpha-1})$ is a real valued function on $\Omega$, so it would be convenient to express $\eta(\varphi_0, \varphi_1)$ as an expectation. For $\alpha\in(0,1)$, we define PDF 
\begin{align}
\begin{split}
\varphi_\alpha(x) :=  \varphi_{0}(x)^{1-\alpha} \,\varphi_{1}(x)^{\alpha} e^{\cd{\varphi_{0}}{\varphi_{1}}},
\quad\text{where}\quad 
\cd{\varphi_0}{\varphi_1} := -\log \int_{\Omega} \varphi_{0}^{1-\alpha} \,\varphi_{1}^{\alpha} d\mu
\end{split}
\label{eq:def:chernoff:div}
\end{align}
is the \emph{Chernoff $\alpha$-divergence} between $\varphi_0$ and $\varphi_1$.  We also define a real valued function
\begin{align}\label{eq:g:alpha}
g_\alpha: \mathbb R \to \mathbb R, \quad g_\alpha(x) := \exp[\min(\alpha x, (\alpha-1) x)] = \min(e^{\alpha x}, e^{(\alpha-1)x}),
\end{align}
Then by direct calculation from \eqref{eq:bayes:min}, we have we have
\begin{align}\label{eq:affinity:exp}
\begin{split}
\eta(\varphi_0, \varphi_1)
= e^{-\cd{\varphi_0}{\varphi_1}}\int_\Omega \varphi_\alpha \min(\lr^{\alpha},\lr^{\alpha-1}) d\mu
=e^{-\cd{\varphi_0}{\varphi_1}}\E_{Y\sim \varphi_\alpha}[g_\alpha(\log l(Y))].
\end{split}
\end{align}
We note that since $g_\alpha(x)\le 1$, we always have $\E_{Y\sim \varphi_\alpha}[g_\alpha(\log l(Y))]\le 1$, which implies $e^{-\cd{\varphi_0}{\varphi_1}}$ is an upper bound of $\eta(\varphi_0, \varphi_1)$. In the last expression, $Y:=(Y_1, \dots, Y_n)$ is a random vector with independent elements on the product space $\Omega$, and one can observe that
\begin{align}\label{eq:decomp:Y}
\begin{split}
Y_j  \sim \varphi_{\alpha j}:= \varphi_{0j}^{1-\alpha} \,\varphi_{1j}^{\alpha} e^{\cd{\varphi_{0j}}{\varphi_{1j}}},
\quad\text{where}\quad 
\cd{\varphi_{0j}}{\varphi_{1j}} := -\log \int_{\Omega_j} \varphi_{0j}^{1-\alpha} \,\varphi_{1j}^{\alpha} d\mu_j.
\end{split}
\end{align}
Let $l_j = {\varphi_{0j}}/{\varphi_{1j}}$, then we can center and decompose $\log l(Y)$ as
\begin{align}\label{eq:decomp:log:l:Y}
\log l(Y) - \E[\log l(Y)]= \sum_{j=1}^n \log l_j(Y_j) - \E[\log l(Y_j)] =:\sum_{j=1}^n Z_j.
\end{align}
$\log l(Y)$ is indeed a summand of independent random variables, so it is approximately normally distributed under some regularization condition, which will be specified in the following theorem.

\begin{theorem}\label{thm:bayes:chernoff}
	We consider the PDFs or PMFs $(\varphi_0, \varphi_1)\in \mathcal F (\Omega, \Sigma, \mu, n)$ defined in \eqref{eq:density:family}, and recall the definitions of $\varphi_\alpha$ in \eqref{eq:def:chernoff:div}, $g_\alpha$ in \eqref{eq:g:alpha}, $\varphi_{\alpha j}$ in \eqref{eq:decomp:Y}, $Y_j$, $Z_j$ in \eqref{eq:decomp:log:l:Y} and $l_j = {\varphi_{0j}}/{\varphi_{1j}}$. We let
	\begin{align*}
	&\alpha^* := \arg\max_{\alpha\in(0,1)} \cd{\varphi_0}{\varphi_1},   \quad
	Y_j\sim \varphi_{\alpha^* j} , \\
	&Z_j = \log l_j(Y_j) - \E[\log l_j(Y_j)] \quad \text{and}  \quad
	\bar \sigma_n := \Big(\f 1n \sum_{j=1}^n \Var[Z_j] \Big)^{1/2}.
	\end{align*}
	If
	${\sum_{j=1}^n \E|Z_j|^3} \le C_1 n \bar\sigma_n^2$, then there exists constant $C_2$ which only depends on $C_1$ such that
	\begin{align*}
	\E_{Y\sim \varphi_{\alpha^*}}[g_{\alpha^*}(\log l(Y))]
	\le  \f{C_2}{\sqrt n\bar \sigma_n (1-\alpha^*)\alpha^*}.
	\end{align*}
	Furthermore, there exists positive constants $C_3$ and $C_4$ which only depend on $C_1$, such that, if $\sqrt n \bar\sigma_n(1-\alpha^*)\alpha^*\ge C_3$, then
	\begin{align*}
	\E_{Y\sim \varphi_{\alpha^*}}[g_{\alpha^*}(\log l(Y))] \ge \f{C_4}{\sqrt n\bar \sigma_n (1-\alpha^*)\alpha^*}.
	\end{align*}
	As a direct consequence of \eqref{eq:affinity:exp},
	\begin{align*}
	\f{C_3}{\sqrt n\bar \sigma_n \alpha^*(1-\alpha^*)} e^{-\cds{\varphi_0}{\varphi_1}}\le \eta(\varphi_0, \varphi_1) \le \f{C_4}{\sqrt n\bar \sigma_n \alpha^*(1-\alpha^*)} e^{-\cds{\varphi_0}{\varphi_1}},
	\end{align*}
	where $\cds{\varphi_0}{\varphi_1}$ is Chernoff information defined in \eqref{eq:chernoff:inf}. By \eqref{eq:tv:affinity}, same bounds holds for $2R(\hat z,z)$.
\end{theorem}

\begin{remark}
	A possible (but not necessarily optimal) choice of $C_2, C_3$ and $C_4$ can be $C_2=2\vee 2(0.56C_1)^{3/2}\exp(\sqrt {2\pi}C_1)$, $C_3=\exp(-2(0.56)\sqrt{2\pi}C_1)/30$, and $C_4=1+0.28C_1$. A gap between $C_3$ and $C_4$ exists in general, which will be shown empirically by simulation in Section~\ref{sec:simulation}.
\end{remark}

\begin{remark}
	The names of Chernoff information/coefficient/divergence in this paper are according to a recent survey \cite{crooks2017measures}. Chernoff information indicates the Chernoff $\alpha$-divergence with $\alpha=\alpha^*$ which maximize $\cd{\varphi_0}{\varphi_1}$. We are not going to calculate the exact value of $\alpha^*$ in this paper. $\cds{\ \cdot\ }{\ \cdot\ }$ is a notation for Chernoff information, and $\alpha^*$ might be different for variant inputs. $\alpha^*$ in Theorem \ref{thm:bayes:chernoff} is unique since we assume $\varphi_0$ and $\varphi_1$ is different on a set with positive measure.
\end{remark}

To gain better understanding of Theorem \ref{thm:bayes:chernoff}, we will derive a corollary of the i.i.d. case. Under such assumptions, many quantities in the theorem become constants. On the other hand, many existing results only consider i.i.d. cases. It is convenient to compare with them with this corollary.

\begin{corollary}[i.i.d. case]\label{cor:bayes:chernoff}
	Let $(\varphi_0^{(n)}, \varphi_1^{(n)})\in \mathcal F (\Omega, \Sigma, \mu, n)$ be a sequence of PMFs or PDFs satisfying  $\varphi_z^{(n)}(x)=\prod_{j=1}^n \bar \varphi_z(x_j)$ for $z\in\{0,1\}$ where $\bar \varphi_0$ and $\bar \varphi_1$ are fixed, then
	\begin{align*}
	\eta(\varphi_0^{(n)}, \varphi_1^{(n)}) = \Theta \Big( \f{1}{\sqrt n} e^{-n\cds{\bar \varphi_0}{\bar \varphi_1}}\Big).
	\end{align*}
\end{corollary}

\begin{proof}
	Under the assumptions, $\E[|Z_j|^3]$ ($Z_j$ defined in \eqref{eq:decomp:log:l:Y}), $\bar\sigma_n$ and $\alpha^*$ are constants, so the assumption
	${\sum_{j=1}^n \E|Z_j|^3} = O(n \bar\sigma_n^2 )$ is satisfied. Moreover, we have $\cds{\varphi_0^{(n)}}{\varphi_1^{(n)}}=n\cds{\bar \varphi_0}{\bar \varphi_1}$ and $\f{1}{\sqrt n\bar \sigma_n \alpha^*(1-\alpha^*)}=\Theta\Big( \frac 1{\sqrt n} \Big)$. $\sqrt n \bar\sigma_n(1-\alpha^*)\alpha^*$ is sufficiently large as $n$ increases. Hence the conclusion is clear by Theorem \ref{thm:bayes:chernoff}.
\end{proof}

{\bf Comparison with existing results.} Chernoff type lower bound can trace back to early literatures.  \cite[Theorem 5]{shannon1967lower} produced the following non asymptotic lower bound for Bayes risk, namely
\begin{align}\label{eq:shannon}
R(\hat z,z)\ge \f 14 \min(e^{-\alpha^*\sqrt n\bar \sigma_n}, e^{ -(1-\alpha^*)\sqrt n \bar{\sigma}_n})e^{-\cds{\varphi_0}{\varphi_1}}.
\end{align}
Since $\min(e^{-\alpha^*\sqrt n\bar \sigma_n}, e^{ -(1-\alpha^*)\sqrt n \bar{\sigma}_n})\ll \f{1}{\sqrt n\bar \sigma_n \alpha^*(1-\alpha^*)}$, \eqref{eq:shannon} is strictly weaker than Theorem \ref{thm:bayes:chernoff}. Despite of its looseness, this lower bound is still actively used in quantum hypothesis testing \cite{dalai2013lower}. For the i.i.d. case, ${\cds{\varphi_0}{\varphi_1}}$ has been shown to be the best achievable exponent \cite[Theorem 11.9.1]{cover2006elements}. This result is restated in \cite[Theorem 2.1]{nussbaum2009chernoff} as follows:
\begin{align*}
\lim_{n\to\infty}\f 1n \log \eta(\varphi_0^{(n)}, \varphi_1^{(n)}) = -\cds{\bar \varphi_0}{\bar \varphi_1},
\end{align*}
under the same condition as Corollary~\ref{cor:bayes:chernoff}. This is obviously true since the exponent of $\eta(\bar \varphi_0, \bar \varphi_1)$ has the form $\log \eta(\varphi_0^{(n)}, \varphi_1^{(n)})=-n\cds{\bar \varphi_0}{\bar \varphi_1}-\f 12 \log n + O(1)$ as we discovered from Corollary \ref{cor:bayes:chernoff}. The $\log n$ term was investigated
in \cite{verdu1986asymptotic}, and applied to hypothesis testing problem in \cite[Lemma 11]{abbe2015community}. However, the result can only apply to Poisson distribution when the samples are i.i.d in a fixed asymptotic setting $\bar{\sigma}_n=\Theta\big(\f{\log n}n\big)$. If the samples are not identical, their lower bound is not valid. \cite{zhou2018optimal} generalized the result to other asymptotic setting; however, their bounds cannot apply to the case when observed data are not i.i.d Poisson distributed. Therefore, neither of them proposed a minimax lower bound that matches their algorithmic error rate in community detection problems.

%

\subsection{Application to Exponential Family}
With a concrete expressions of $\varphi_0$ and $\varphi_1$, we can write the Chernoff type bound in Theorem \ref{thm:bayes:chernoff} with a closed form up to the choice of $\alpha^*$ and a constant.
In this section, we are interested in the exponential family with PMF or PDF of the form
\begin{align}\label{eq:exp:fam}
\varphi(x;\theta) = h(x)\exp[\theta^\top T(x) - A(\theta)].
\end{align}
where $A$ is a smooth function defined on a convex set in certain Euclidean space. Let us we assume $\{\varphi_{zj}: z\in\{0,1\}, j\in [n]\}$ is contained in exponential family. To be specific, we assume there exist parameters $\theta_{zj}$, $z\in\{0,1\}, j\in[n]$ such that
\begin{align}\label{eq:pzj:exp}
\varphi_{zj}(x_j) =p(x;\theta_{zj})= h(x_j)\exp[\theta_{zj}^\top T(x_j)-A(\theta_{zj})]
\end{align}
We still define $\varphi_0$ and $\varphi_1$ as in \eqref{eq:def:prod:space} on some measure space such that $(\varphi_0, \varphi_1)\in \mathcal F (\Omega, \Sigma, \mu, n)$ (see \eqref{eq:density:family}).
Let us define $\theta_{\alpha j} := (1-\alpha)\theta_{0j} + \alpha\theta_{1j}$, then $\theta_{\alpha j}$ is a valid parameter since we assume the parameter space is convex. The Chernoff $\alpha$-divergence has a close form
\begin{align}
\cd{\varphi_{0j}}{\varphi_{1j}}
= (1-\alpha)A(\theta_{0j})+\alpha A( \theta_{1j})-A(\theta_{\alpha j}).\label{eq:chernoff:inf:exp:fam}
\end{align}
See Section \ref{sec:proof:chernoff:inf:exp:fam} for derivation. Suppose $Y_j\sim \varphi_{\alpha^* j}$, then using the definition of $Z_j$ in \eqref{eq:decomp:log:l:Y},
\begin{align}\label{eq:exp:fam:z:j}
Z_j = \log l_j(Y_j) - \E[\log l_j(Y_j)] = (\theta_{0j}-\theta_{1j})^\top (T(Y_j) - \nabla A(\theta_{\alpha j})).
\end{align}
We have $\Var[Z_j] = (\theta_{0j}-\theta_{1j})^\top {\bf H}(A(\theta_{\alpha^*j}))(\theta_{0j}-\theta_{1j})$ where ${\bf H}(A(\theta))$ is the Hessian matrix of $A$ evaluated at $\theta$. Now we can establish a corollary when $\varphi_{zj}$'s belong to exponential family.
\begin{corollary}[exponential family]\label{cor:exp:fam}
	Under the same assumptions in Theorem \ref{thm:bayes:chernoff}, assuming $\varphi_{zj}$'s have the form \eqref{eq:pzj:exp}, and let
	\begin{align*}
	\bar \sigma_n := \Big(\f 1n \sum_{j=1}^n \Var[Z_j] \Big)^{1/2}
	= \Big(\f 1n \sum_{j=1}^n (\theta_{0j}-\theta_{1j})^\top {\bf H}(A(\theta_{\alpha^*j}))(\theta_{0j}-\theta_{1j})\Big)^{1/2} .
	\end{align*}
	Suppose
	${\sum_{j=1}^n \E|Z_j|^3} \le C_1n \bar\sigma_n^2$, using the same constants $C_2, C_3$ and $C_4$ in Theorem \ref{thm:bayes:chernoff}, then
	\begin{align*}
	\eta(\varphi_0, \varphi_1) \le \Big( \f{C_2}{\sqrt n\bar \sigma_n (1-\alpha^*)\alpha^*} e^{-\summ j n[(1-\alpha^*)A(\theta_{0j})+\alpha^* A( \theta_{1j})-A(\theta_{\alpha^* j})]}\Big),
	\end{align*}
	If $\sqrt n\bar \sigma_n (1-\alpha^*)\alpha^*\ge C_3$, then we have
	\begin{align*}
	\eta(\varphi_0, \varphi_1) \ge \Big( \f{C_4}{\sqrt n\bar \sigma_n (1-\alpha^*)\alpha^*} e^{-\summ j n[(1-\alpha^*)A(\theta_{0j})+\alpha^* A( \theta_{1j})-A(\theta_{\alpha^* j})]}\Big).
	\end{align*}
\end{corollary}

Most of the time, we are interested in the problem that if the testing rule or clustering algorithm achieves the optimal bound, and we want the bound to be as precise as possible. Under the exponential family assumption, the upper and lower bound in this corollary has a closed form up to the choice of $\alpha^*$. The value of $\alpha^*$ is not important in theoretical analysis because it is usually relatively stable as $n$ increases under some regularization conditions. In many special cases, $\alpha^*=1/2$. See the simulation section  (Section~\ref{sec:simulation}) for examples. Now we will show the advantages of these bounds in concrete examples and the community detection problem.



\subsection{An Example of Bernoulli Distribution}\label{sec:example:bern}
We are going to investigate a special case in exponential family.  Let $p_{zj}\in (0,1)$ and $\theta_{zj} =  \log \big( \f{p_{zj}}{1-p_{zj}} \big)$ for $z\in \{0,1\}, j\in[n]$, and define PMFs of $\text{Bern}(p_{zj})$:
\begin{align*}
\varphi_{zj}(x) := \varphi(x, \theta_{zj}) :=
\begin{cases}
p_{zj}, &\text{if } x = 1,\\
1-p_{zj}, &\text{if } x = 0,\\
0, &\text{otherwise.}
\end{cases}
\end{align*}
It coincides with \eqref{eq:pzj:exp} if we let $A(\theta) = \log(1+e^\theta)$, $T(x) = x$ and $h(x)=1_{\{0,1\}}(x)$, i.e. $
\varphi_{zj}(x)
=h(x)\exp[\theta_{zj}^\top T(x)-A(\theta_{zj})]$.
Let us briefly recall the testing problem in Section \ref{sec:general:cases}. We randomly draw a number $z\in\{0,1\}$ with equal probability $1/2$, and draw a random sample $X = \{X_1,\dots,X_n\}$ where $X_j\sim \text{Bern}(p_{zj})$ independently. As usual, we want to recover $z$ given $X=x\in \{0,1\}^n$. Then we have
\begin{align}\label{eq:chernoff:bern}
e^{-\cd{\varphi_{0}}{\varphi_{1}}}=\prod_{j=1}^n[{p_{0j}^{1-\alpha} p_{1j}^\alpha + (1-p_{0j})^{1-\alpha} (1-p_{1j})^\alpha}].
\end{align}
Let $p_{\alpha j} = \f{p_{0j}^{1-\alpha} p_{1j}^\alpha}{p_{0j}^{1-\alpha} p_{1j}^\alpha + (1-p_{0j})^{1-\alpha} (1-p_{1j})^\alpha}$ and recall the definition of $\alpha^*$ and $\bar \sigma_n$ from Theorem \ref{thm:bayes:chernoff}, then we have the identity
\begin{align}\label{eq:variance:z}
n\bar{\sigma}_n^2 = \sum_{j=1}^n \Big[ \log \f{p_{0j}(1-p_{1j})}{p_{1j}(1-p_{0j})}  \Big]^2 p_{\alpha^* j}(1-p_{\alpha^* j}).
\end{align}
Now let us apply Theorem \ref{thm:bayes:chernoff}. Suppose $\max_{j\in[n]} \Big| \log \f{p_{0j}(1-p_{1j})}{p_{1j}(1-p_{0j})} \Big| \le C_1$, then there exists constants $C_2, C_3$ and $C_4$ which only depend on $C_1$, such that if $\sqrt{n}\bar\sigma_n\alpha^*(1-\alpha^*)\ge C_2$, then the upper and lower bound of $\eta(\varphi_0,\varphi_1)=2R(\hat z, z)$ in the theorem holds. \par
Finally, it is worth mentioning a special case when $p_{z1}=\dots = p_{zn}:=\bar p_{z}$ for $z\in\{0,1\}$. Let $\psi_z(x)$ be the PMF of $\text{Bin}(n, \bar p_{z})$. Then one can check that
\begin{align}\label{eq:bern:bin:equiv}
\eta(\psi_0, \psi_1) = \eta(\varphi_0, \varphi_1).
\end{align}
This is due to the fact that, given the observed data $x$, the optimal test only relies on the sufficient statistics $\sum_{j=1}^n x_j$. This observation can generalize Corollary \ref{cor:bayes:chernoff} to the cases when only the sufficient statistics, which is a summand of i.i.d. random variables, are observed. For example, one can apply Corollary \ref{cor:bayes:chernoff} to Bayes error probability of Poisson parameter testing as the Poisson variable is the sum of arbitrarily many i.i.d Poisson variables.\par
Calculations details in this section will appear in Section \ref{proof:eq:chernoff:bern} to  \ref{proof:bern:bin:equiv}.

\section{Application to Community Detection in SBM}\label{sec:community}

The results in previous section can apply to almost all clustering and classification problem in statistics. A typical example is community detection in stochastic block model. Given some good estimates of the parameters, community detection is indeed a classification problem. Hence the clustering error rate of the label estimates heavily depends on the Bayes error probability.

\subsection{Background of SBM}\label{sec:background:sbm}

We will be working on a network which can be represented by a symmetric adjacency matrix $A\in\{0,1\}^{n\times n}$, where the nodes are indexed by $[n]$. We assume that there are $\Kr$ communities on $[n]$, and the membership of the nodes are given by $z\in[\Kr]^\nr$. Thus $z_i=k$ if node $i\in[n]$  belongs to community $k\in[\Kr]$. We let $n_k:=|\{i: z_i = k\}|$ be the size of $k$th community. Under a stochastic block model (SBM), given a symmetric probability matrix $P\in[0,1]^{\Kr\times\Kr}$,
\begin{align}\label{eq:def:adjacency:matrix}
A_{ij} = A_{ji} \sim \text{Bern}(P_{z_iz_j}) \text{ for all } i>j \text{ independently},
\end{align}
and $A_{ii}=0$ for all $i\in[n]$. That is, the connectivity of nodes only depends on their memberships, and there are no self-loops. A fundamental task of community detection on SBM is to recover $z$ given $A$ and $K$. To have notation consistent with the previous section, we define
\begin{align}\label{eq:parameter:vector}
p_{kj}:=P_{kz_j}  \quad \text{for } k\in[K],  j\in[n].
\end{align}
In other words, if $z_i=k$, $\E[A_{ij}]=p_{kj}$ whenever $j\ne i$. Thus, the vectors $p_{k*}$ and $\E[A_{i*}]$ are the same at all entries but the $i$th one.  When $n$ is large, the effect of one entry is up to a constant. Here $p_{k*} = (p_{k1}, \dots, p_{kn})$, and similarly $A_{i*}$ is the $i$th row of $A$. This notation will be used in the rest of this paper. We will consider the parameters satisfies
\begin{align}\label{eq:abs:constants}
n_k\in\big[ \f{n}{\beta K},\f{\beta n}K \big], \quad
\max_{k,\ell}P_{k\ell} := p^*\le 1-\eps, \quad
\text{and}\quad
\f{\max_{k,\ell}P_{k\ell}}{\min_{k',\ell'}P_{k'\ell'}}\le \omega
\end{align}
for some fixed constants $\beta>1, \eps\in(0,1)$, and $\omega>1$.  We will assume $K\le K^*$ where $K^*$ is also a fixed constant, i.e., $K = O(1)$. $\beta$ controls the balance between different communities. There are no too small or too large communities. All connectivity probabilities are bounded above by $1-\eps$, which is a mild sparsity assumption.
For estimate $\hat z$ of $z$, we are interested in the error rate defined as
\begin{align}
\mis(\hat z, z) = \min_{\pi\in S_K} \f 1n \summ i n 1\{\pi(\hat z) \ne z\}
\end{align}
where $S_K$ is the symmetric group contains all permutations of $[K]$ and the permutation $\pi$ will apply entry-wisely on $\hat z$.

\subsection{Fundamental Limit}

Let us first consider a simplified symmetric hypothesis testing problem in the SBM. In the community detection problem describe in the previous section, only the adjacency matrix $A$ and number of community $K$ is given. Now suppose additionally, we know $z_{-i}$, i.e., the all labels but the $i$th one, and the connectivity matrix $P$, our goal is to recover $z_i$. To further simplify the problem, we assume $z_i\in\{k,\ell\}$, then the hypothesis problem becomes comparison between the parameters $p_{k*}$ and $p_{\ell*}$ defined in \eqref{eq:parameter:vector}. 
Since the distribution of Bernoulli vector $A_{i*}$ can be characterized by $p_{k*}$ if $z_i=k$, we will write
\begin{align}\label{eq:short:cd}
\begin{split}
\cd{p_{k*}}{p_{\ell*}} &:=D_\alpha\Big({\bigotimes_{j=1}^n\text{Bern}(p_{kj})}\|{\bigotimes_{j=1}^n\text{Bern}(p_{\ell j})}\Big);\\
\eta(p_{k*}, p_{\ell*})&:=
\eta\Big({\bigotimes_{j=1}^n\text{Bern}(p_{kj})}, {\bigotimes_{j=1}^n\text{Bern}(p_{\ell j})}\Big).
\end{split}
\end{align}
which also denote the same quantities if we input the corresponding PMF's. Substituting $p_{0*}$ and $p_{1*}$ with $p_{k*}$ and $p_{\ell *}$ in Section \ref{sec:example:bern}, and using the assumptions about SBM in Section \ref{sec:background:sbm}, we have the following lemma using the previous result.

\begin{lemma}\label{lem:fund:limit}
	Given adjacency matrix $A$ and parameters $K$, $z_{-i}$ and $P$, and knowing that $z_i=k$ or $\ell$ with probability $1/2$, then Bayes estimator $\hat z_i$
	\begin{align*}
	\hat z_i:=\arg\max_{k\in[K]} \sum_{j\ne i} A_{ij}\log p_{kj} + (1-A_{ij})\log (1-p_{kj})
	\end{align*}
	satisfies $\P(\hat z\ne z) = \f 12 \eta(p_{k*}, p_{\ell*})$. Assuming \eqref{eq:abs:constants}, there  exist constants $C_1$ and $C_2$ only depends on $\beta, \eps, K^*$ and $\omega$ such that, if $C_1\le np^*\le C_2 (D^*)^2$,
	\begin{align*}
	\P(\hat z\ne z)=\Theta\Big( \Big( \sqrt{np^*}\max_{j\in[n]}\Big|\log\f{p_{kj}(1-p_{\ell j})}{p_{\ell j}(1-p_{kj})} \Big| \Big)^{-1}\exp(-\cds{p_{k*}}{p_{\ell*}})\Big).
	\end{align*}
\end{lemma}


Now we will derive a minimax lower bound of community detection problem. 
We will consider the following parameter space:
\begin{align}\label{eq:par:space:strict}
\begin{split}
{\mathcal S}(n,p^*, K, &D^*) := \Big\{(z,P): z\in[K]^n, P\in(0,1)^{K\times K}, P=P^\top,\\ &\eqref{eq:abs:constants} \text{ is satisfied, } \infnorm{P} = p^*, \min_{k\ne\ell} \max_{r\in[K]} \frac{P_{kr}}{P_{\ell r}} \vee \frac{P_{\ell r}}{P_{k r}}\ge \omega',\\ 
&p_{kj}=P_{kz_j} , \min_{k\ne \ell} \cds{p_{k*}}{ p_{\ell*}}\ge D^*\Big\}.
\end{split}
\end{align}
Here, we introduce a new constant $\omega'>1$ to provide a lower bound of $\max_{j\in[n]} \Big| \log \f{p_{kj}(1-p_{\ell j})}{p_{\ell j}(1-p_{kj})}  \Big|$. This is a mild assumption since as long as $P_{k*}\ne P_{\ell*}$ for all $k,\ell\in[K]$, then $\omega'>1$ exists.

\begin{theorem}[minimax lower bound]\label{thm:minimax:lower:bound}
	Let $\mathcal S = \mathcal S(n, p^*, K, D^*)$ in \eqref{eq:par:space:strict}. Suppose $3\le K\le K^*$, $np^*$ sufficiently large, and $D^*$ satisfies there exists $q$ satisfies $0<q<p^*\le 1-\eps$ and  $\omega'\le p^*/q\le \omega$ such that $D^*= \cds{\text{Bin}(\lfloor n/K \rfloor, p^*)}{\text{Bin}(\lfloor n/K \rfloor, q)}$, we have
	\begin{align*}
	\inf_{\hat z} \sup_{{\mathcal S}} \E[\mis(\hat z,z)] = \Omega\Big( \frac 1{\sqrt{np^*}} e^{-D^*} \Big),
	\end{align*}
	where the big $\Omega$ notation only involves constant depends on $\beta, \eps, K^*, \omega$, and $\omega'$.
\end{theorem}

The proof will appear in Section~\ref{proof:thm:minimax:lower:bound}, which is inspired by several previous works \cite{zhang2016minimax, gao2018community}. The existing minimax lower bound $\exp(-(1+o(1)) D^*))$ is provided under the assortative assumptions (i.e., ``$p$ vs $q$", see introduction for details).   The existence of $q$ in the theorem help reduce to the assortative cases; however, additional proving technique is needed for such more general setting and removing the term $o(1)D^*$. Our proof cannot extend to the case $K=2$. We will leave this as future work.

\subsection{Algorithm Achieving the Minimax Lower Bound}

Our algorithm is inspired by the pseudo-likelihood approach in \cite{amini2013pseudo}.
We define an operator to estimate $P$ according to adjacency matrix $A$ and estimated labels $\tilde z$:
\begin{align}\label{eq:operator:B}
\mathcal B(A,\tilde z):= (\hat P_{k\ell})\in [0,1]^{K\times K}, \quad
\hat P_{k\ell} := \f{\sum_{i>j} A_{ij} 1\{\tilde z_i = k, \tilde z_j = \ell\}}{\sum_{i>j} 1\{\tilde z_i = k, \tilde z_j = \ell\}}.
\end{align}
We will also use likelihood ratio classifier defined as follows:
\begin{align}\label{eq:operator:L}
\begin{split}
\mathcal L(A, \hat P, \tilde z):= (\hat z_i)\in [K]^n, \quad
\hat z_i = \arg \max_{k\in[K]} \sum_{j\ne i} A_{ij}\log \hat P_{k\hat z_j} + (1-A_{ij})\log (1-\hat P_{k \hat z_j}).
\end{split}
\end{align}
Note that we can apply these two operators on submatrice of $A$ with the corresponding indices if needed. One can observe that this is an EM-type algorithm if we repeat \eqref{eq:operator:B} and \eqref{eq:operator:L} iteratively, i.e., \eqref{eq:operator:B} is the expectation step and \eqref{eq:operator:L} is the maximization step. As pointed out in \cite{zhou2018optimal}, it requires at least two iterations of EM-type update to achieve the optimal error rate up to a constant. To generate enough independence between iterations, we combine the block partition method in \cite{chin2015stochastic} and ``leave-one-out" trick in \cite{gao2017achieving}. It is worth noting that besides the dependence between $A$ and $\tilde z$ in $\mathcal L(A,\hat P, \tilde z)$, other dependence can be handled by uniform bounds. Details about Algorithm~\ref{alg:lrc} will be describe as follows:

Step~\ref{step:1} to~\ref{step:2}: We apply spectral clustering on the whole adjacency matrix. However, we will only use its output in the matching step and approximate an initial estimate $\tilde P$ of $P$. The dependence between $\tilde P$ and $A$ can be handled by uniform bounds.\par
Step~\ref{step:3}: This is the block partitioning trick. Data in different blocks will be used in different steps to acquire independence. \par
Step~\ref{step:4} to~\ref{step:5}: This is the ``leave-one-out" trick. In each iteration, we only use the data of the $j$th node in step~\ref{step:10}, so the last likelihood ratio classifier will be independent with other steps in the for loop.\par
Step~\ref{step:6} to~\ref{step:7}: We apply spectral clustering on two of the subblocks. The labels will be consistent after we match them with $\tilde z$. Note that although $\tilde z$ depends on $A$, $\tilde z'_{I'}$ and $\tilde z'_{J'}$ only depend on the corresponding subblocks as long as the spectral clustering algorithm outputs good enough labels.\par
Step~\ref{step:8}: We apply the first likelihood ratio classifier on a different subblock from ones used in step~\ref{step:6}.\par
Step~\ref{step:9}: We obtain new estimate $\hat P$ of $P$ according to updated labels in step~\ref{step:8}.\par
Step~\ref{step:10}: We update the label again according to the new $\hat P$ and $\tilde z'$ obtained in step~\ref{step:9}. \par
Step~\ref{step:12} to~\ref{step:16}: A spectral clustering algorithm proposed in \cite{zhou2018spectral}. We refer readers to see the regularization step in the original paper. \par
Step~\ref{step:17} to~\ref{step:20}: A matching algorithm finding the optimal permutation between labels. Its a linear assignment problem with computational complexity $O(K^3)$ \cite{jonker1987shortest}.\par


\begin{algorithm}[t]
	\caption{Community detection}
	\begin{algorithmic}[1]
		\State\textbf{Input:} Adjacency matrix $A$, number of communities $K$.
		\State\textbf{Output:} Estimated labels $\hat z$.
		\State $\tilde z\gets \SC(A,K)$.
		\label{step:1}
		\State $\tilde P\gets\mathcal B(A,\tilde z)$.
		\label{step:2}
		\State Let $I\subset[n]$ with $|I|=\lfloor n/2\rfloor$ be a random subset of indices. Let $J=[n]\backslash I$.
		\label{step:3}
		\For{$j=1$ to $n$}
		\label{step:4}
		\State $I'\gets I\backslash\{j\}$, $J'\gets J\backslash \{j\}$.
		\label{step:5}
		\State $\tilde z_{I'}' \gets \SC(A_{I'\times I'},K)$, $\tilde z_{J'}' \gets \SC(A_{J'\times J'},K)$.
		\label{step:6}
		\State $\tilde z'_{I'} \gets \match(\tilde z_{I'}, \tilde z'_{I'})$, $\tilde z'_{J'}\gets\match(\tilde z_{J'}, \tilde z'_{J'})$.
		\label{step:7}
		\State $\tilde z'_{I'} \gets \mathcal L(A_{I'\times J'}, \tilde P, \tilde z'_{J'})$, $\tilde z'_{J'} \gets \mathcal L(A_{J'\times I'}, \tilde P, \tilde z'_{I'})$.
		\label{step:8}
		\State $\tilde z' \gets (\tilde z_i, \tilde z'_{I'}, \tilde z'_{J'})$, $\hat P\gets\mathcal B(A, \tilde z')$.
		\label{step:9}
		\State $\hat z_j \gets \mathcal L(A_{j*},\hat P, \tilde z')$.
		\label{step:10}
		\EndFor
		\Function{SC}{$A, K$}
		\label{step:12}
		\State Apply degree-truncation to $A$ to obtain $\Are$.
		\State Apply SVD on $\Are$ so that $\Are=U\Sigma U^T$. Let $\hat \Sigma$ contains top $K$ singular values on the diagonal and $\hat U$ contains corresponding singular vectors.
		\State Output the K-means clustering result on the rows of $\hat U\hat\Sigma$.
		\EndFunction
		\label{step:16}
		\Function{\match}{$\tilde z, z$}
		\label{step:17}
		\State $\tilde z \gets \arg\min_{\pi(\tilde z):\pi\in S_K} \sum_{i=1}^n 1\{z_i\ne \pi(z_i)\}$.
		\State Output $\tilde z$.
		\EndFunction
		\label{step:20}
	\end{algorithmic}
	\label{alg:lrc}
\end{algorithm}

The following block matrix might help understand partitioning of adjacency matrix $A$ in the algorithm.
\renewcommand{\arraystretch}{1.2}
\begin{align*}
A=
\left[
\begin{array}{ c | c c | c c }
0 &  \multicolumn{4}{c}{A_{i\times(I'\cup J')}} \\
\cline{1-5}
\multirow{4}{*}{\vdots}  & & &  &  \\
& \multicolumn{2}{c|}{\raisebox{.6\normalbaselineskip}[0pt][0pt]{$A_{I'\times I'}$}} & \multicolumn{2}{c}{\raisebox{.6\normalbaselineskip}[0pt][0pt]{$A_{I'\times J'}$}}\\
\cline{2-5}
&  &  & & \multicolumn{1}{c}{} \\
& \multicolumn{2}{c|}{\raisebox{.6\normalbaselineskip}[0pt][0pt]{$A_{J'\times I'}$}} & \multicolumn{2}{c}{\raisebox{.6\normalbaselineskip}[0pt][0pt]{$A_{J'\times J'}$}}
\end{array}
\right], \quad
\left[
\begin{array}{ c | c c | c c }
0 &  \multicolumn{4}{c}{\mbox{2nd LR (step \ref{step:10})}} \\
\cline{1-5}
\multirow{4}{*}{\vdots}  & & &  &  \\
& \multicolumn{2}{c|}{\raisebox{.6\normalbaselineskip}[0pt][0pt]{\makecell{2nd SC\\(step \ref{step:6})}}} & \multicolumn{2}{c}{\raisebox{.6\normalbaselineskip}[0pt][0pt]{\makecell{1st LR\\(step \ref{step:8})}}}\\
\cline{2-5}
&  &  & & \multicolumn{1}{c}{} \\
& \multicolumn{2}{c|}{\raisebox{.6\normalbaselineskip}[0pt][0pt]{\makecell{1st LR\\(step \ref{step:8})}}} & \multicolumn{2}{c}{\raisebox{.6\normalbaselineskip}[0pt][0pt]{\makecell{2nd SC\\(step \ref{step:6})}}}
\end{array}
\right]
\end{align*}
Note that ``$\ \vdots\ $" represents the block $A_{i\times(I'\cup J')}^\top$. We can see that the second spectral clustering and both likelihood ratio tests are applied on different blocks of the adjacency matrix, so we do not need to worry about dependence between steps. Now we present the theoretical guarantees of the output of Algorithm~\ref{alg:lrc}.

\begin{theorem}\label{thm:alg:mis:upper:bound}
	Let us assume \eqref{eq:abs:constants} and $K\le K^*$ for some fixed constants $\beta, \omega$, $\eps$, and $K^*$. We also briefly denote $D^*:=\min_{k\ne\ell} \cds{p_{k*}}{p_{\ell*}}$ and $\eta^*=\max_{k\ne\ell}\eta(p_{k*}, p_{\ell*})$. Suppose the spectral clustering algorithm returns $\tilde z$ satisfying $\mis(\tilde z ,z)= O((D^*)^{-1})$ with probability at least $1-n^{-(r+1)}$ for some $r>0$, then there exist constants $C_1$ and $C_2$ only depends on $\beta, \eps, K^*$ and $\omega$ such that, if $C_1\le np^*\le C_2 (D^*)^2$, then the output from Algorithm \ref{alg:lrc} satisfies:
	\begin{itemize}
		\item [(a)] If $D^*\le (1+r/2)\log n$, then $\E[\mis(\hat z,z)] = O(\eta^*)$.
		\item [(b)] For all $s>0$, if $D^*\ge (1+2s)\log n$, then $\mis(\hat z,z)\ge \eta^*$ with probability $O\Big( \f {1}{n^{r\wedge s}} \Big)$.
	\end{itemize}
	Moreover, $\eta^*$ can be replaced by $$\max_{k\ne\ell}\Big( \sqrt{np^*}\max_{j\in[n]}\Big|\log\f{p_{kj}(1-p_{\ell j})}{p_{\ell j}(1-p_{kj})} \Big| \Big)^{-1}\exp(-\cds{p_{k*}}{p_{\ell*}})$$ in statements above.
\end{theorem}

\begin{remark}
	The assumption about spectral clustering algorithm is a theorem in \cite{zhou2018spectral}. We would like to refer readers to the original paper for details.
\end{remark}

\begin{remark}
	Suppose we further assume that $\min_{k\ne\ell} \max_{r\in[K]} \frac{P_{kr}}{P_{\ell r}} \vee \frac{P_{\ell r}}{P_{k r}}\ge \omega'$ for some constant $\omega'>1$, then the error bound achieve the minimax lower bound in Theorem~\ref{thm:minimax:lower:bound}. The consistency result are presented in two regimes, but they overlap since the choice of $s$ in (b) is arbitrary.
\end{remark}

\begin{remark}[Comparison with existing results.]
	We have already compared some results in literature. Here, we will summarize the novelty in details.
	\begin{itemize}
		\item[1.] Existing paper either consider the asymptotic behavior of optimal community detection in general undirected or bipartite SBM \cite{abbe2015community, zhou2018optimal} or symmetric assortative SBM \cite{chin2015stochastic, gao2017achieving}. We extend the minimax theory and algorithms to general directed SBM in a non-asymptotic setup.
		\item[2.] We apply twice local updates (likelihood ratio tests) on symmetric adjacency matrix in our algorithm. It is also possible to apply multiple times by partitioning more blocks. Although multiple step of local updates are allowed in \cite{zhang2017theoretical} by variational inference, data splitting method is required and lacking in their algorithm. 
		\item[3.] Thanks to the advanced Chernoff bound in Theorem~\ref{thm:bayes:chernoff}, we provide sharpened minimax error rate and tight misclassification rate for our algorithm. In particular, we replace the uncertain term $\exp(o(1)D^*)$ in \cite{zhang2016minimax} by an explicit expression.
	\end{itemize}
\end{remark}

\section{Simulation}\label{sec:simulation}

We will show that, in some asymptotic setting, the Bayes error probability converges with a rate expected in Theorem~\ref{thm:bayes:chernoff} by simulation. Let us consider Bernoulli distributions analyzed in Section~\ref{sec:example:bern}. Let $p_{01} = p_{02} = \dots  = p_{0n} = p_{1(n+1)} = p_{1(n+2)} = \dots = p_{1(2n)} = 0.55$, and $p_{11} = p_{12} = \dots  = p_{1n} = p_{0(n+1)} = p_{0(n+2)} = \dots = p_{0(2n)} = 0.45$, i.e.,
\begin{align*}
p_{0*} &= (\underbrace{0.55, 0.55, \dots, 0.55}_\text{n \text{times}}, \underbrace{0.45, 0.45, \dots, 0.45}_\text{n \text{times}});\\
p_{1*} &= (\underbrace{0.45, 0.45, \dots, 0.45}_\text{n \text{times}}, \underbrace{0.55, 0.55, \dots, 0.55}_\text{n \text{times}}).
\end{align*}
Now we consider Chernoff $\alpha$-divergence. The optimal $\alpha^*$ in Theorem~\ref{thm:bayes:chernoff} is 1/2 by symmetry. Using the notation in \eqref{eq:short:cd}, by \eqref{eq:chernoff:bern}, we have
\begin{align*}
e^{-\cds{p_{0*}}{p_{1*}}} = (2\sqrt{0.55\cdot 0.45})^n.
\end{align*}
By \eqref{eq:variance:z} with some details in Section~\ref{proof:eq:chernoff:bern}, we have
\begin{align*}
n\bar{\sigma}_n^2 = \sum_{j=1}^n \Big( \log \f{0.55^2}{0.45^2}  \Big)^2 0.5(1-0.5) = \Theta(n).
\end{align*}
By Theorem~\ref{thm:bayes:chernoff}, we expect
\begin{align*}
\eta(p_{0*}, p_{1*}) = \Theta\Big( \f{1}{\sqrt n} (2\sqrt{0.55\cdot 0.45})^n \Big).
\end{align*}
Or equivalently, $$a_n:=\log \eta(p_{0*}, p_{1*}) - n\log (2\sqrt{0.55\cdot 0.45}) $$
asymptotically behaves like $\f 12 \log n+C$. We can also think of $p_{0*}$ and $p_{1*}$ as the parameters in 
\eqref{eq:parameter:vector} associated with SBM with community sizes $n_0=n_1=n$ and connectivity matrix
$$P=
\begin{bmatrix}
0.55 &0.45\\
0.45 &0.55
\end{bmatrix}
.$$
By Theorem~\ref{thm:alg:mis:upper:bound}, we expect $$b_n:=\log (2\cdot\text{Misclassification rate}) - n\log (2\sqrt{0.55\cdot 0.45}) $$ also tends to $\f 12 \log n+C$. Note that 2 comes from the fact that $\eta(p_{0*}, p_{1*})=2\cdot (\text{Bayes error probability})$ would help the simulation scale better. We will use the true Bernoulli PMF to compute $\eta(p_{0*}, p_{1*})$, then find the misclassification rate of Algorithm~\ref{alg:lrc} and compute $b_n$.
\begin{center}
	\includegraphics[scale=0.5]{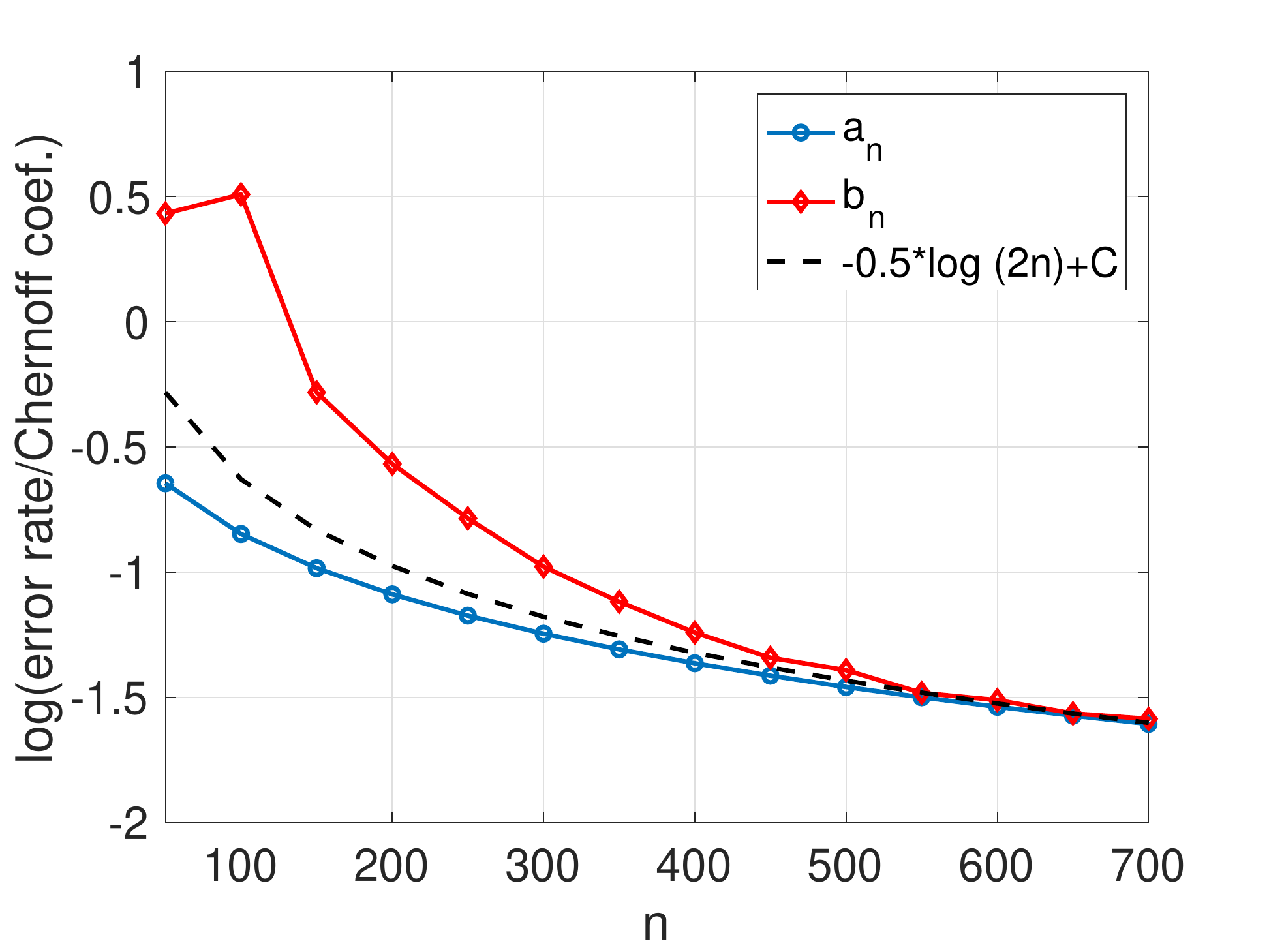}
\end{center}

From the plot, we observe that $n$ increases, both $a_n$ and $b_n$ converge to the same constant as expected. For smaller $n$, the misclassification rate is large since initialization in Algorithm~\ref{alg:lrc} is not accurate enough; however, $b_n$ becomes stable when $n$ getting large.

Another interesting empirical result we want to show by simulation is that, the constant involved in big $\Theta$ notation in Corollary~\ref{cor:bayes:chernoff} does not converge in general. We let $p_{0*} = 0.3\cdot \ones_n$ and $p_{1*} = 0.7 \cdot \ones_n$, i.e.,
\begin{align*}
p_{0*} = (\underbrace{0.3, 0.3, \dots, 0.3}_\text{n \text{times}})
\quad\text{and}\quad
p_{1*} = (\underbrace{0.7, 0.7, \dots, 0.7}_\text{n \text{times}} ).
\end{align*}
By symmetry, we have $\alpha^*=1/2$, so 
\begin{align*}
\eta(p_{0*}, p_{1*}) = \Theta\Big( \f{1}{\sqrt n} (0.3\cdot 0.7)^{n/2} \Big).
\end{align*}
Again, we let $a_n = \log \eta(p_{0*}, p_{1*}) - \f n2\log (0.3\cdot 0.7)$. The following plots shows the behavior of $a_n$.
\begin{center}
	\includegraphics[scale=0.45]{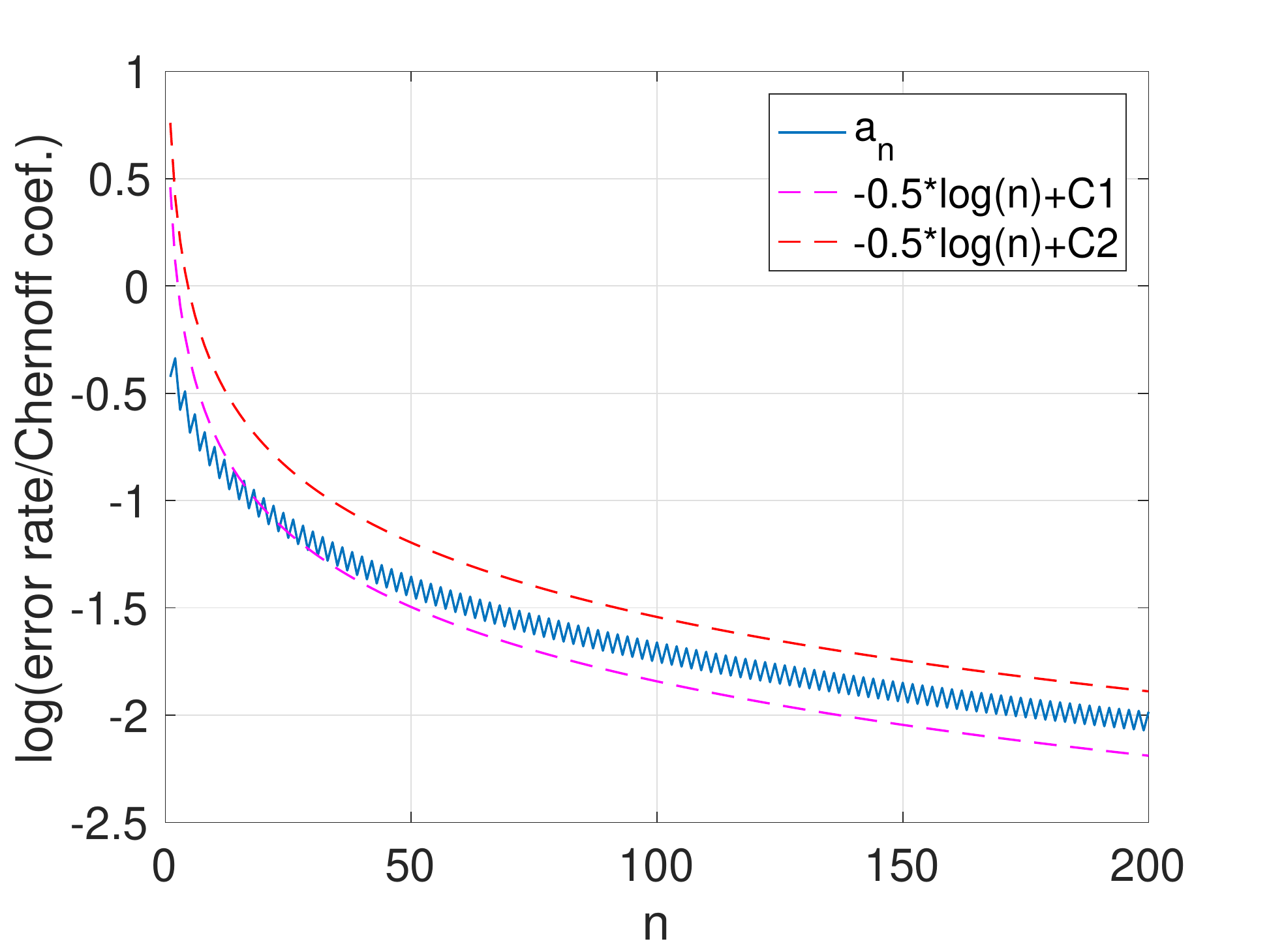}\quad
	\includegraphics[scale=0.45]{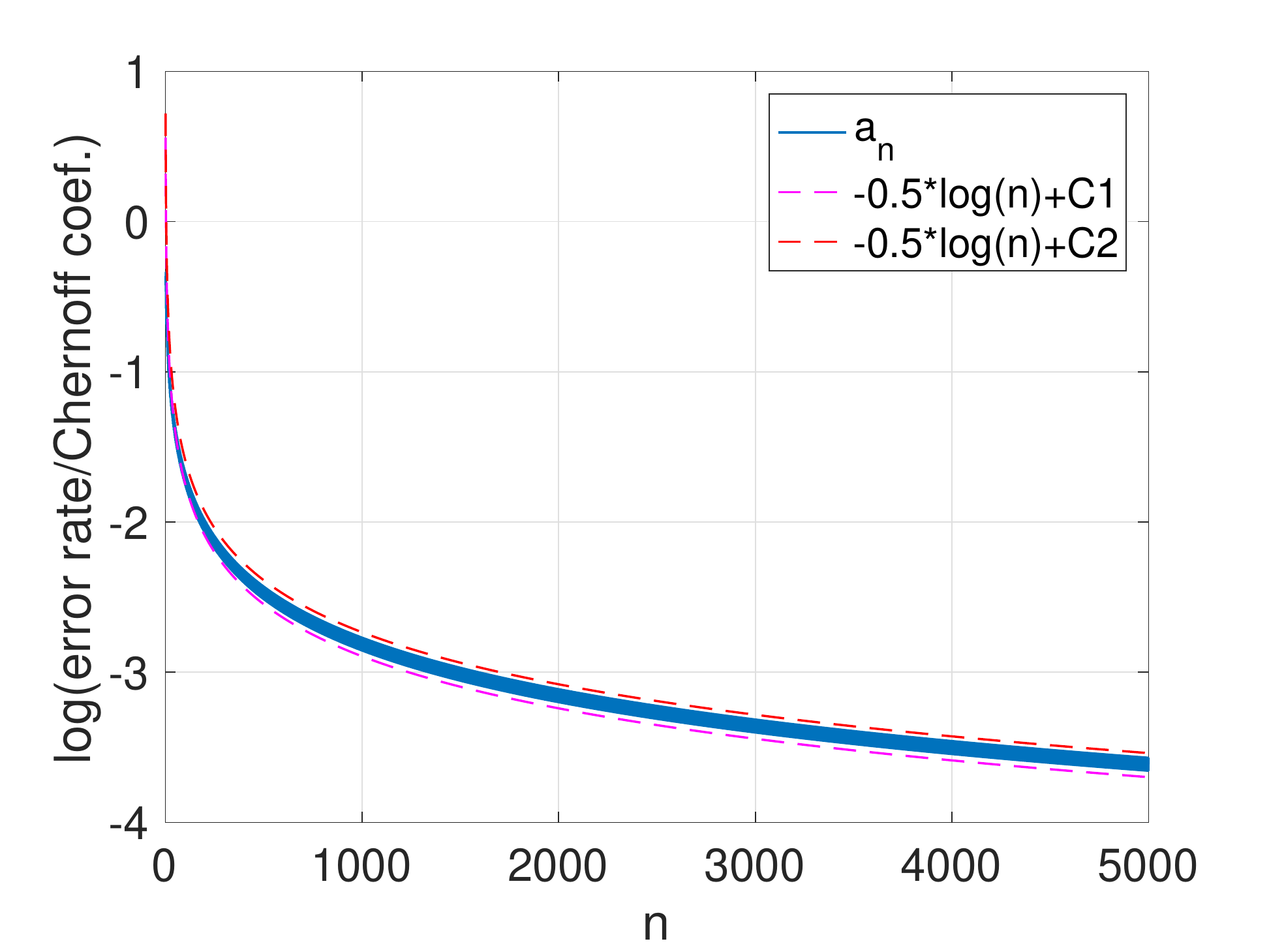}
\end{center}
They are plots of $a_n$ in different ranges of $n$. Although $a_n$ asymptotically behaves like $\f 12 \log n$, it oscillates up and down until infinity. This simulation result empirically shows that, 
$a_n - \f 12 \log n$ does not converge to any constant for such $p_{0*}$ and $p_{1*}$.

\section{Proofs of Section~\ref{sec:chernoff}}\label{sec:proofs}

\subsection{Proof of Theorem \ref{thm:bayes:chernoff}}

\begin{lemma}\label{lem:alpha:mean:zero}
	We recall $\alpha^* = \arg\max_{\alpha\in(0,1)} \cd{\varphi_0}{\varphi_1}$ and $Y\sim \varphi_{\alpha^*}$ from the assumption of the theorem, then $\E[\log l(Y)]=0$.
\end{lemma}

\begin{proof}
	By definition of $Y$, we have
	\begin{align*}
	\E[\log l(Y)] &= \int_\Omega \varphi_{\alpha^*} \log l(y) d\mu(y) = e^{-\cds{\varphi_0}{\varphi_1}} \int_\Omega \varphi_0^{1-\alpha_*} \varphi_1^{\alpha_*} \log \f{\varphi_1}{\varphi_0} d\mu
	\end{align*}
	Recall that we assume the Kullback–Leibler divergence $\kld{\varphi_0}{\varphi_1}$ and $\kld{\varphi_1}{\varphi_0}$ exist, so for $\alpha\in(0,1), \big|\varphi_0^{1-\alpha}\varphi_1^\alpha \log \f{\varphi_1}{\varphi_0}\big| \le \big|(\varphi_0+\varphi_1)\log \f{\varphi_1}{\varphi_0}\big|$ is integrable. By mean value theorem and dominated convergence theorem,
	\begin{align}
	\int_\Omega \varphi_0^{1-\alpha} \varphi_1^{\alpha} \log \f{\varphi_1}{\varphi_0} d\mu = \int_\Omega \f{d}{d\alpha} \varphi_0^{1-\alpha}\varphi_1^\alpha d\mu =\f{d}{d\alpha}\int_\Omega \varphi_0^{1-\alpha}\varphi_1^\alpha d\mu.
	\end{align}
	Since $\alpha\mapsto -\varphi_0(x)^{1-\alpha} \varphi_1(x)^\alpha$ is convex for $x\in\Omega$, $\alpha\mapsto -\int_\Omega \varphi_0^{1-\alpha} \varphi_1^\alpha d\mu$ is also convex, and it is indeed strictly convex if $\varphi_0\ne \varphi_1$ on a set with nonzero measure. Therefore, $\cd{\varphi_0}{\varphi_1}$ achieves maximum if and only if $\f{d}{d\alpha}\int_\Omega \varphi_0^{1-\alpha}\varphi_1^\alpha d\mu=0$, which is true if we evaluate at $\alpha = \alpha^*$. Hence $\E[\log l(Y)]  = 0$.
\end{proof}

\begin{prop}\label{prop:norm:cdf}
	Let $\Phi$ be the cumulative distribution function of standard normal distribution, then for $x>0$,
	\begin{align}\label{eq:normal:cdf:infinity}
	\f 1x-\f 1{x^3}\le \sqrt{2\pi}e^{x^2/2}\Phi(-x)  \le \f 1x.
	\end{align}
	and
	\begin{align}\label{eq:normal:cdf:zero}
	\sqrt{2\pi} e^{x^2/2}\Big( \Phi(x)-\f 12 \Big)\ge x+\f {x^3}3.
	\end{align}
	In particular,
	\begin{align}\label{eq:normal:cdf:zero:part}
	\Phi(x)\ge \frac 12 + \f 1{\sqrt{2\pi}} \Big( x-\f{2x^3}3 \Big).
	\end{align}
\end{prop}
\begin{proof}
	For $x>0$, we have divergent series expanded at $\infty$:
	\begin{align*}
	\sqrt{2\pi}e^{x^2/2}\Phi(-x) = \f 1x + \sum_{i=1}^\infty \f{(-1)^i(2i-1)!}{2^{i-1}(i-1)!} \f 1{x^{2i+1}}
	= \f 1x - \f 1{x^3} + \f 3{x^5} - \f {15}{x^7} + \dots
	\end{align*}
	which implies \eqref{eq:normal:cdf:infinity}. We also have the power series expanded at $0$:
	\begin{align*}
	\sqrt{2\pi} e^{x^2/2}\Big( \Phi(x)-\f 12 \Big)
	= \sum_{i=0}^\infty \f{x^{2i+1}}{(2i+1)!!} = x+ \f {x^3}3 + \f{x^5}{3\cdot 5} + \f{x^7}{3\cdot 5\cdot 7}+ \dots
	\end{align*}
	which implies \eqref{eq:normal:cdf:zero}. Then we expand $e^{-x^2/2}\Big( x+\f {x^3}3 \Big)$ and have
	\begin{align*}
	e^{-x^2/2}\Big( x+\f {x^3}3 \Big)=x-\f {2x^3}3 + \f{7x^5}{30}-\dots
	\end{align*}
	Then we obtain \eqref{eq:normal:cdf:zero:part}.
\end{proof}

\begin{lemma}\label{lemma:normal:case}
	Recall from \eqref{eq:g:alpha} that $g_\alpha(x)=\exp(\min(\alpha x, (\alpha-1)x))$. Suppose $Z\sim\mathcal N(0,\sigma^2)$, then
	\begin{align*}
	\f 1{\sqrt{2\pi}\sigma \alpha (1-\alpha)}-\f 1{\sqrt{2\pi}\sigma^3 \alpha^3 (1-\alpha)^3}\le \E[g_\alpha(Z)] \le \f 1{\sqrt{2\pi}\sigma \alpha (1-\alpha)}.
	\end{align*}
\end{lemma}

\begin{proof}
	We have
	\begin{align*}
	\E[g_\alpha(Z)] &= \E[\exp(\min(\alpha Z, (\alpha-1)Z))]\\
	&=\f 1{\sqrt{2\pi}\sigma}\int_{-\infty}^0 e^{\alpha x} e^{-\f{x^2}{2\sigma^2}} dx + \f 1{\sqrt{2\pi}}\int_{0}^\infty e^{(\alpha-1)x} e^{-\f{x^2}{2\sigma^2}}dx\\
	&=\f {e^{\sigma^2\alpha^2/2}}{\sqrt{2\pi}\sigma}\int_{-\infty}^0  e^{-\f{(x-\sigma^2\alpha)^2}{2\sigma^2}} dx + \f {e^{\sigma^2(1-\alpha)^2/2}}{\sqrt{2\pi}\sigma}\int_{0}^\infty e^{-\f{(x-\sigma^2(1-\alpha))^2}{2\sigma^2}}dx\\
	& = \exp\Big( \f{\sigma^2 \alpha^2}2\Big)\Phi(-\sigma \alpha)+\exp\Big( \f{\sigma^2 (1-\alpha)^2}2\Big)\Phi(-\sigma (1-\alpha))
	\end{align*}
	By \eqref{eq:normal:cdf:infinity}, we have
	\begin{align*}
	\E[g_\alpha(Z)]\le \f 1{\sqrt{2\pi}\sigma \alpha}+\f 1{\sqrt{2\pi}\sigma (1-\alpha)} = \f 1{\sqrt{2\pi}\sigma \alpha (1-\alpha)},
	\end{align*}
	and
	\begin{align*}
	\E[g_\alpha(Z)]&\ge \f 1{\sqrt{2\pi}\sigma \alpha}-\f 1{\sqrt{2\pi}\sigma^3 \alpha^3}+\f 1{\sqrt{2\pi}\sigma (1-\alpha)} - \f 1{\sqrt{2\pi}\sigma^3 (1-\alpha)^3}\\
	&\ge \f 1{\sqrt{2\pi}\sigma \alpha (1-\alpha)}-\f 1{\sqrt{2\pi}\sigma^3 \alpha^3 (1-\alpha)^3}.
	\end{align*}
\end{proof}

\begin{theorem}[Berry-Esseen theorem]\label{thm:berry:esseen}
	Let $\{Z_i\}_{i=1}^n$ be independent random variables with zero means and $\E[\sum_{i=1}^n Z_i^2]=\sigma^2$. Let $F$ be the distribution function of $\sum_{i=1}^n Z_i/\sigma$, then there exists an absolute constant $C_0\le 0.56$ such that for every $x\in\reals$,
	\begin{align}
	|F(x)-\Phi(x)|\le \f{C_0\sum_{i=1}^n\E[|Z_i|^3] } {\sigma^3}.
	\end{align}
\end{theorem}

\begin{proof}
	The best upper bound of $C_0$ so far is given by \cite{shevtsova2010improvement}. We would also like to refer readers to see a proof by Stein's method in \cite{chen2001non}.
\end{proof}

\begin{proof}[Proof of Theorem \ref{thm:bayes:chernoff}]
	Let $\alpha = \alpha^*$, then by Lemma \ref{lem:alpha:mean:zero}, $\E[\sum_{j=1}^n Z_j] =\E[\log l(Y)]= 0$. Let us define $\sigma^2 = \E[\sum_{j=1}^n Z_j^2] $ as in Theorem \ref{thm:berry:esseen}. Note that $\sigma = \sqrt n \bar\sigma_n$. By assumption ${\sum_{j=1}^n \E|Z_j|^3} \le C_1 n \bar\sigma_n^2$, the distribution function $F$ of $\summ j n Z_j/\sigma$ satisfies for $x\in\reals$, $|F(x)-\Phi(x)|\le \f C\sigma$ where $C:=0.56C_1$. Recall that $\log l(Y) = \summ j n Z_j$, we have
	\begin{align*}
	\E[g_\alpha(\log l(Y))]
	&= \int_0^1  \P(g_\alpha(\log l(Y))>x)dx\\
	&= \int_0^1 \P \Big( \f{\log x}{\alpha \sigma} < \f{\log l(Y)}{\sigma} < \f{\log x}{(\alpha -1)\sigma} \Big) dx\\
	& = \int_0^1 F\Big( \f{\log x}{(\alpha -1)\sigma}\Big) - F\Big( \f{\log x}{\alpha\sigma} \Big) dx \\
	&\le \int_0^1\Phi\Big( \f{\log x}{(\alpha -1)\sigma}\Big) - \Phi\Big( \f{\log x}{\alpha\sigma} \Big) dx+\f{2C}\sigma\\
	&=\f 1{\sqrt{2\pi}\sigma \alpha(1-\alpha)} +\f{2C}\sigma
	\le \f {1+C/2}{\sigma \alpha(1-\alpha)}.
	\end{align*}
	On the other hand, for any $t\in [0,1]$,
	\begin{align*}
	\E[g_\alpha(\log l(Y))]
	& = \int_0^1 F\Big( \f{\log x}{(\alpha -1)\sigma}\Big) - F\Big( \f{\log x}{\alpha\sigma} \Big) dx\\
	&\ge \int_0^t \Phi\Big( \f{\log x}{(\alpha -1)\sigma}\Big) - \Phi\Big( \f{\log x}{\alpha\sigma} \Big) -
	\f{2C}\sigma dx\\
	&= \int_0^t \Phi\Big( \f{\log x}{(\alpha -1)\sigma}\Big) - \Phi\Big( \f{\log x}{\alpha\sigma} \Big) dx-
	\f{2tC}\sigma.
	\end{align*}
	By Fubini's theorem,
	\begin{align}\label{eq:fubini}
	\begin{split}
	&\int_0^t \Phi\Big( \f{\log x}{\sigma(\alpha -1)}\Big) - \Phi\Big( \f{\log x}{\sigma\alpha} \Big) dx
	=\int_0^t \frac 1{\sqrt{2\pi}} \int_{\f{\log x}{\sigma(\alpha -1)}}^{\f{\log x}{\sigma\alpha}}e^{-y^2/2}dydx\\
	&=\f 1{\sqrt{2\pi}}\Big[\int_{-\infty}^{\f{\log t}{\sigma(\alpha-1)}}e^{\sigma(\alpha-1)y-\f{y^2}{2}}dy
	+\int_{-\f{\log t}{\sigma\alpha}}^{\infty}e^{-\sigma\alpha x-\f{y^2}{2}}dy\Big] 
	+t\Big[ \Phi\Big( \f{\log t}{\sigma(\alpha-1)} \Big) - \Phi\Big( \f{\log t}{\sigma\alpha} \Big) \Big].
	\end{split}
	\end{align}
	The second integral in the last line can be evaluated as
	\begin{align*}
	\f 1{\sqrt{2\pi}}\int_{-\f{\log t}{\sigma\alpha}}^{\infty}e^{-\sigma\alpha y-\f{y^2}{2}}dy
	=\f 1{\sqrt{2\pi}}\int_{-\f{\log t}{\sigma\alpha}}^{\infty}e^{\f{\sigma^2\alpha^2}2-\f{(y+\sigma\alpha)^2}{2}}dy
	={e^{\f{\sigma^2\alpha^2}2}} \Phi\Big( \f{\log t}{\sigma\alpha} -\sigma\alpha \Big).
	\end{align*}
	For the first integral, we similarly have
	\begin{align*}
	\f 1{\sqrt{2\pi}}\int_{-\infty}^{\f{\log t}{\sigma(\alpha-1)}}e^{\sigma(\alpha-1)y-\f{y^2}{2}}dy
	={e^{\f{\sigma^2(1-\alpha)^2}2}} \Phi\Big( \f{\log t}{\sigma(1-\alpha)} -\sigma(1-\alpha) \Big).
	\end{align*}
	Assuming $\sigma\alpha(1-\alpha)\ge \sqrt{2\pi}C\vee 2$ and let $t =\exp[-2\sqrt{2\pi} C(1-\alpha)\alpha]$, using $\alpha(1-\alpha)\le 1/4$, we have
	\begin{align*}
	-\f{\log t}{\sigma\alpha}\le \f{2\sqrt{2\pi} C\alpha(1-\alpha)}{\sqrt{2\pi}C}\le \f 12
	\quad \text{and}\quad
	\sigma\alpha-\f{\log t}{\sigma\alpha}\le \sigma\alpha+\f 12\le \f{5\sigma\alpha}4.
	\end{align*}
	By \eqref{eq:normal:cdf:infinity} and the fact that the function $\f 1x - \f 1{x^3}$ is decreasing on $[\sqrt 3, \infty]$, we have
	\begin{align*}
	{e^{\f{\sigma^2\alpha^2}2}} \Phi\Big( \f{\log t}{\sigma\alpha} -\sigma\alpha \Big)
	&\ge \f 1{\sqrt{2\pi}}\exp\Big({\f{\sigma^2\alpha^2-\big( \f{\log t}{\sigma\alpha} -\sigma\alpha \big)^2}2}\Big)\Big[ \f 1{\sigma\alpha-\f{\log t}{\sigma\alpha} } -\f 1{\Big(\sigma\alpha-\f{\log t}{\sigma\alpha}\Big)^3 }\Big]\\
	&\ge \f 1{\sqrt{2\pi}}\exp\Big( \log t - \f 12 \Big( \f{\log t}{\sigma\alpha} \Big)^2 \Big)\Big[ \f 4{5\sigma\alpha} -  \f {64}{125\sigma^3\alpha^3} \Big]\\
	&\ge \f 1{\sqrt{2\pi}}\exp(-2\sqrt{2\pi} C (1-\alpha)\alpha-1/8)\Big[ \f 4{5\sigma\alpha} -  \f {64}{125(4\sigma\alpha)} \Big]\\
	&\ge \f{\exp(-2\sqrt{2\pi}C)}{5\sigma\alpha}.
	\end{align*}
	Similarly, we have
	\begin{align*}
	{e^{\f{\sigma^2(1-\alpha)^2}2}} \Phi\Big( \f{\log t}{\sigma(1-\alpha)} -\sigma(1-\alpha) \Big)\ge \f{\exp(-2\sqrt{2\pi}C)}{5\sigma(1-\alpha)}.
	\end{align*}
	Hence the integral in \eqref{eq:fubini} has lower bound
	\begin{align}\label{eq:two:integrals}
	\begin{split}
	\f 1{\sqrt{2\pi}}\Big[&\int_{-\infty}^{\f{\log t}{\sigma(\alpha-1)}}e^{\sigma(\alpha-1)y-\f{y^2}{2}}dy
	+\int_{-\f{\log t}{\sigma\alpha}}^{\infty}e^{-\sigma\alpha x-\f{y^2}{2}}dy\Big]\\
	&\ge \f{\exp(-2\sqrt{2\pi}C)}{5\sigma\alpha}+\f{\exp(-2\sqrt{2\pi}C)}{5\sigma(1-\alpha)}
	=\f{\exp(-2\sqrt{2\pi}C)}{5\sigma\alpha(1-\alpha)}.
	\end{split}
	\end{align}
	Now we consider another term $\Phi\Big( \f{\log t}{\sigma(\alpha-1)} \Big)$ in \eqref{eq:fubini}. By \eqref{eq:normal:cdf:zero:part}, we have
	\begin{align*}
	\Phi\Big( \f{\log t}{\sigma(\alpha-1)} \Big)&\ge \f 12 + \f 1{\sqrt{2\pi}}\Big( \f{\log t}{\sigma(\alpha-1)} -\f 23\Big( \f{\log t}{\sigma(\alpha-1)} \Big)^3 \Big)\\
	&=\f 12 + \f 1{\sqrt{2\pi}}\Big( \f{-2\sqrt{2\pi}C\alpha(1-\alpha)}{\sigma(\alpha-1)} -\f 23\Big( \f{-2\sqrt{2\pi}C\alpha(1-\alpha)}{\sigma(\alpha-1)} \Big)^3 \Big)\\
	&=\f 12 + \f {2C\alpha}\sigma-\f {32\pi C^3 \alpha^3}{3\sigma^3}
	\end{align*}
	and similarly the term $-\Phi\Big( \f{\log t}{\sigma\alpha} \Big)$ has lower bound
	\begin{align*}
	-\Phi\Big( \f{\log t}{\sigma\alpha} \Big)&= \Phi\Big( -\f{\log t}{\sigma\alpha}\Big) - 1
	\ge \f 12 + \f {2C(1-\alpha)}\sigma-\f {32\pi C^3 (1-\alpha)^3}{3\sigma^3}-1.
	\end{align*}
	Therefore,
	\begin{align*}
	t\Big[ \Phi\Big( \f{\log t}{(\alpha-1)\sigma} &\Big) - \Phi\Big( \f{\log t}{\alpha\sigma} \Big) \Big]
	-\f{2tC}\sigma
	\ge \f{2tC}\sigma -\f {32\pi C^3 \alpha^3}{3\sigma^3}-\f {32\pi C^3 (1-\alpha)^3}{3\sigma^3}-\f{2tC}\sigma
	\ge -\f {32\pi C^3 }{3\sigma^3}.
	\end{align*}
	Assuming $\sigma\alpha(1-\alpha)\ge 2C^{3/2}\exp(\sqrt{2\pi}C)$, then $\sigma^2\alpha^2(1-\alpha)^2\ge 4C^3\exp(2 \sqrt{2\pi}C)$, so we have
	\begin{align}\label{eq:last:lower}
	\begin{split}
	\f {32\pi C^3 }{3\sigma^3}
	\le\f {32\pi C^3 \alpha^3(1-\alpha)^3}{3\sigma^3\alpha^3(1-\alpha)^3}
	\le \f {\pi C^3 }{6\sigma^3\alpha^3(1-\alpha)^3}
	\le \f {\pi\exp(-2 \sqrt{2\pi}C)}{24\sigma\alpha(1-\alpha)}
	\le \f {\exp(-2 \sqrt{2\pi}C)}{6\sigma\alpha(1-\alpha)}.
	\end{split}
	\end{align}
	We combine \eqref{eq:two:integrals} and \eqref{eq:last:lower} and have, $$\E[g_\alpha(\log l(Y))] \ge \f {\exp(-2 \sqrt{2\pi}C)}{5\sigma\alpha(1-\alpha)}-\f {\exp(-2 \sqrt{2\pi}C)}{6\sigma\alpha(1-\alpha)}=\f {\exp(-2 \sqrt{2\pi}C)}{30\sigma\alpha(1-\alpha)}.$$
\end{proof}

\subsection{Proof of \eqref{eq:chernoff:inf:exp:fam}}\label{sec:proof:chernoff:inf:exp:fam}

\begin{align*}
\cd{\varphi_{0j}}{\varphi_{1j}} &= -\log \int_{\Omega_j} \varphi_{0j}^{1-\alpha} \,\varphi_{1j}^{\alpha} d\mu\nonumber\\
&=-\log \int_{\Omega_j} h(x) \exp\{ [(1-\alpha)\theta_{0j}+\alpha\theta_{1j}]^\top T(x) - (1-\alpha)A(\theta_{0j})-\alpha A(\theta_{1j})\}\nonumber\\
&=-\log \Big\{\exp[-(1-\alpha)A(\theta_{0j})-\alpha A( \theta_{1j})+A(\theta_{\alpha j}))] \int_{\Omega_j} \varphi(x;\theta_{\alpha j}) dx\Big\}\nonumber\\
& = (1-\alpha)A(\theta_{0j})+\alpha A( \theta_{1j})-A(\theta_{\alpha j}).
\end{align*}

\subsection{Proof of variance of \eqref{eq:exp:fam:z:j}}

The variance of $Z_j$ can be directly derived from the following Proposition. Its proof is skipped for brevity.

\begin{prop}
	A random variable $X\sim \varphi(x;\theta)$ in \eqref{eq:exp:fam} satisfies:
	\begin{enumerate}
		\item [(a)] The moment generating function of $T(X)$, $M_{T(X)}(t) = \exp[A(\theta+t)-A(\theta)]$ if it exists.
		\item [(b)] 
		$\E[T(X)] = \nabla A(\theta)$ and $\Var[T(X)] = {\bf H}(A(\theta))$ where ${\bf H}(A(\theta))$ is the Hessian matrix of $A$ evaluated at $\theta$.
	\end{enumerate}
\end{prop}

\subsection{Proof of \eqref{eq:chernoff:bern}} \label{proof:eq:chernoff:bern}

For $\alpha\in(0,1)$, we recall $\theta_{\alpha j} = (1-\alpha)\theta_{0j} + \alpha\theta_{1j}$ and define
\begin{align*}
p_{\alpha j} &:= \f 1{1+e^{-\theta_{\alpha j}}}=\f 1{1+e^{-(1-\alpha)\theta_{0j} - \alpha\theta_{1j}}}\\
&=\f 1{1+\Big( \f{1-p_{0j}}{p_{0j}} \Big)^{1-\alpha}\Big( \f{1-p_{1j}}{p_{1j}} \Big)^{\alpha}}
= \f{p_{0j}^{1-\alpha} p_{1j}^\alpha}{p_{0j}^{1-\alpha} p_{1j}^\alpha + (1-p_{0j})^{1-\alpha} (1-p_{1j})^\alpha}.
\end{align*}
It is worth noting that the identities still holds when $\alpha=0$ or $1$, including the next one.
\begin{align*}
\exp(A(\theta_{\alpha j}))=1+e^{\theta_{\alpha j}} = 1+\f{p_{\alpha j}}{1-p_{\alpha j}} = \f{1}{1-p_{\alpha j}}
= \f{p_{0j}^{1-\alpha} p_{1j}^\alpha + (1-p_{0j})^{1-\alpha} (1-p_{1j})^\alpha}{(1-p_{0j})^{1-\alpha} (1-p_{1j})^\alpha}.
\end{align*}
By \eqref{eq:chernoff:inf:exp:fam}, we can write the Chernoff coefficient for $p_{0j}$ and $p_{1j}$ in terms of $p_{\alpha j}$'s:
\begin{align*}
e^{-\cd{\varphi_{0j}}{\varphi_{1j}}}&=e^{-(1-\alpha)A(\theta_{0j})-\alpha A( \theta_{1j})+A(\theta_{\alpha j})} \\
&={(1-p_{0j})^{1-\alpha} (1-p_{1j})^\alpha}\f{p_{0j}^{1-\alpha} p_{1j}^\alpha + (1-p_{0j})^{1-\alpha} (1-p_{1j})^\alpha}{(1-p_{0j})^{1-\alpha} (1-p_{1j})^\alpha}\\
&={p_{0j}^{1-\alpha} p_{1j}^\alpha + (1-p_{0j})^{1-\alpha} (1-p_{1j})^\alpha}.
\end{align*}
Hence,
\begin{align*}
e^{-\cd{\varphi_{0}}{\varphi_{1}}} &= e^{-\summ j n\cd{p_{0j}}{p_{1j}}} = \prod_{j = 1}^n [p_{0j}^{1-\alpha} p_{1j}^\alpha + (1-p_{0j})^{1-\alpha} (1-p_{1j})^\alpha].
\end{align*}

\subsection{Proof of \eqref{eq:variance:z}}
We recall the definition of $Y_j$ from \eqref{eq:decomp:Y}, and have 
$Y_j\sim \text{Bern}(p_{\alpha j})$. 
By \eqref{eq:exp:fam:z:j}, and letting $\alpha=\alpha^*$,
\begin{align*}
Z_j &= (Y_j - p_{\alpha^* j})\log \f{p_{0j}(1-p_{1j})}{p_{1j}(1-p_{0j})}
\end{align*}
with variance
\begin{align*}
\Var[Z_j] &= (\theta_{0j}-\theta_{1j})^2 \Var(Y_j) =\Big[ \log \f{p_{0j}(1-p_{1j})}{p_{1j}(1-p_{0j})} \Big]^2 p_{\alpha j}(1-p_{\alpha j}).
\end{align*}
Summing over $j\in[n]$ of $\Var[Z_j]$ gives \eqref{eq:variance:z}.
\subsection{The rest of proof of Bernoulli example in Section \ref{sec:example:bern}}
To show the upper and lower bound, it remains to show that $\max_{j\in[n]} \Big| \log \f{p_{0j}(1-p_{1j})}{p_{1j}(1-p_{0j})} \Big|  \le C_1$ is a sufficient condition of Theorem \ref{thm:bayes:chernoff}. We have
\begin{align*}
\E[|Z_j|^3] &= \Big| \log \f{p_{0j}(1-p_{1j})}{p_{1j}(1-p_{0j})}  \Big|^3
[(1-p_{\alpha j})p_{\alpha j}^3+p_{\alpha j}(1-p_{\alpha j})^3]
\le \Big| \log \f{p_{0j}(1-p_{1j})}{p_{1j}(1-p_{0j})}  \Big|^3 p_{\alpha j}(1-p_{\alpha j}).
\end{align*}
Suppose $\max_{j\in[n]} \Big| \log \f{p_{0j}(1-p_{1j})}{p_{1j}(1-p_{0j})} \Big| \le C_1$, then
\begin{align*}
\f{\sum_{i=1}^n \E[|Z_j|^3]}{\sum_{i=1}^n\Var[Z_j]}
=\f {\sum_{i=1}^n \Big| \log \f{p_{0j}(1-p_{1j})}{p_{1j}(1-p_{0j})}  \Big|^3 p_{\alpha j}(1-p_{\alpha j})}{\sum_{i=1}^n\Big[ \log \f{p_{0j}(1-p_{1j})}{p_{1j}(1-p_{0j})}  \Big]^2 p_{\alpha j}(1-p_{\alpha j})}
\le \max_{j\in[n]} \Big| \log \f{p_{0j}(1-p_{1j})}{p_{1j}(1-p_{0j})}  \Big| \le C_1.
\end{align*}
This implies ${\sum_{j=1}^n \E|Z_j|^3} \le C_1 n \bar\sigma_n^2$.

\subsection{Proof of \eqref{eq:bern:bin:equiv}}\label{proof:bern:bin:equiv}
We recall the definition of total variation affinity $\eta$ from \eqref{eq:tv:affinity}, and have
\begin{align*}
\eta(\varphi_0, \varphi_1)
&= \sum_{x\in \{0,1\}} \min\Big(\prod_{j=1}^n \bar p_0^{x_j}(1-\bar p_0)^{1-x_j}, \prod_{j=1}^n \bar p_1^{x_j}(1-\bar p_1)^{1-x_j}\Big)\\
&=\sum_{x\in \{0,1\}} \min\Big(\prod_{j=1}^n \bar p_0^{\sum_{j=1}^n x_j}(1-\bar p_0)^{n-\sum_{j=1}^n x_j}, \prod_{j=1}^n \bar p_1^{\sum_{j=1}^n x_j}(1-\bar p_1)^{n-\sum_{j=1}^n x_j}\Big)\\
&=\sum_{y=0}^n {n\choose y} \min(\bar p_0^y(1-\bar p_0)^{n-y}, \bar p_1^y(1-\bar p_1)^{n-y})\\
&=\eta(\psi_0, \psi_1).
\end{align*}

\section{Proofs of Section~\ref{sec:community}}\label{sec:proofs:sec:3}

\subsection{Proof of Lemma \ref{lem:fund:limit}}

It suffices to check the assumptions in Theorem \ref{thm:bayes:chernoff} are satisfied. We use the notation in Section \ref{sec:example:bern}, and replace $p_{0*}$ and $p_{1*}$ by $p_{k*}$ and $p_{\ell*}$. Firstly,
\begin{align*}
\f{\sum_{i=1}^n \E[|Z_j|^3]}{\sum_{i=1}^n\Var[Z_j]}
\le \max_{j\in[n]} \Big| \log \f{p_{kj}(1-p_{\ell j})}{p_{\ell j}(1-p_{kj})}  \Big| \le \Big| \log \f\omega\eps\Big|.
\end{align*}
Secondly, by Lemma \ref{lem:frac:lower:bound}, $\sqrt{np^*}\max_{j\in[n]}\Big|\log\f{p_{kj}(1-p_{\ell j})}{p_{\ell j}(1-p_{kj})} \Big| $ is sufficiently large. Furthermore, under \eqref{eq:abs:constants},
\begin{align*}
\max_{j\in[n]} \Big[ \log \f{p_{kj}(1-p_{\ell j})}{p_{\ell j}(1-p_{kj})}  \Big]^2 \f{\eps np^*}{\omega\beta K^*}
&\le\sum_{j=1}^n \Big[ \log \f{p_{kj}(1-p_{\ell j})}{p_{\ell j}(1-p_{kj})}  \Big]^2 p_{\alpha^* j}(1-p_{\alpha^* j})\\
&\le \max_{j\in[n]} \Big[ \log \f{p_{kj}(1-p_{\ell j})}{p_{\ell j}(1-p_{kj})}  \Big]^2 np^*.
\end{align*}
Finally, by Lemma \ref{lem:bounded:alpha}, we can remove $\alpha^*(1-\alpha^*)$ since it is bounded below by constant and bounded above by $1/4$.

\subsection{Proof of Theorem \ref{thm:minimax:lower:bound}}\label{proof:thm:minimax:lower:bound}

Let $I\in[n]$ be a fixed subset of indices and $\tilde z_I$ be a fixed label vector satisfies
\begin{align*}
|\{i\in I: \tilde z_i=k\}|=
\begin{cases*}
\lfloor \big( 1-\f 1{\big( 17\vee \f{2}{\beta-1} \big)\beta} \big) \f nK \rfloor, &\text{ if } $k=1$;\\
\lfloor \f nK \rfloor, &\text{ otherwise }.
\end{cases*}
\end{align*}
Let $Z_I(\tilde z_I) = \{z\in [K]^n: z_I = \tilde z_I,z_{I^c}\in\{1,2\}^{|I^c|} \} $ where $I^c = [n]\backslash I$. We consider a subset of $\mathcal S$:
\begin{align*}
\mathcal S'&:= \mathcal S'(q,\tilde z_I) \\
&:= \Big\{(z,P)\in\mathcal S: P_{k\ell} =
\begin{cases*}
p^*, \text{ if } k=\ell>3 \text{ or } (k,\ell) = (2,3) \text{ or } (3,2);\\
q, \text{ otherwise. }
\end{cases*},
z_I \in Z_I
\Big\}.
\end{align*}
Let us briefly write $p:=p^*$ in the rest of proof. In $\mathcal S'$, we restrict the connectivity matrix to be
$$
P = \left[
\begin{matrix}
q &q    &q & q&\dots &q\\
q &q    &p & q&\dots &q\\
q &p    &q & q&\dots &q\\
q &q    &q & p&\dots &q\\
\vdots &\vdots &\vdots&\vdots &\ddots &\vdots\\
q &q &q &q&\dots &p
\end{matrix}
\right]
$$
When we consider the minimax risk $\inf_{\hat z} \sup_{\mathcal S'} \E[\mis(\hat z,z)]$ over $\mathcal S'(P, \tilde z_I)$, $P$ and $\tilde z_I$ are known, i.e., $\hat z = \hat z(P,\tilde z_I)$. First of all, we want to show that the optimal strategy must satisfy $\hat z\in Z_I$. We note that $|I^c|\le \f{n}{17\beta K} + K\le \f{n}{16\beta K}$ when $n$ is sufficiently large. For any $\tilde z\in Z_I$, we have $\mis(\tilde z, z)\le \f{n}{16\beta K}$. Now we consider two cases.
\begin{itemize}
	\item[1.] If $\mis(\tilde z, \hat z)\ge \f n{4\beta K}$, then $\mis(z,\hat z)\ge \f n{4\beta K}- \f{n}{16\beta K}=\f{3n}{16\beta K}$. If we let $\hat z' = \tilde z$ be a new strategy, it is  uniformly better than the original strategy.
	\item[2.] If $\mis(\tilde z, \hat z)< \f n{4\beta K}$, then the optimal permutation $\sigma^*$ of the $\mis$ function is unique, that is, $\mis(\tilde z, \hat z) = \f 1n\sum_{i=1}^n 1\{\sigma^*(\hat z_{i})\ne \tilde z\}$. Given $\tilde z_I = z_I$, we define new $\hat z'$ such that
	\begin{align*}
	\hat z_i' =
	\begin{cases}
	\tilde z_i, & \text{if $i\in I$;} \\
	\sigma^*(\hat z_i), & \text{otherwise.}
	\end{cases}
	\end{align*}
	Then $\hat z' \in Z_I$ is a strategy has error rate not higher than $\z_i$.
\end{itemize}
Therefore, for every strategy $\hat z$, there exists $\hat z'\in Z_I$ such that the misclassification rate is not higher than $\hat z$, so we can assume the optimal strategy $\hat z\in Z_I$. Hence we have
\begin{align*}
\inf_{\hat z} \sup_{\mathcal S'} \E[\mis(\hat z,z)]
\ge \inf_{\hat z} \sup_{\mathcal S'} \E[\mis(\hat z,z)]
=\inf_{\hat z\in Z_I} \sup_{\mathcal S'} \f 1n\E[H(\hat z,z)].
\end{align*}
where $H$ indicates the hamming distance defined as $H(\hat z,z)=\sum_{i=1}^n 1\{\hat z_i\ne z_i\}$. Now we consider the error rate on the undetermined indices $I^c$. First, we want to show that for any choice of $z_{I^c}$, $z=(z_{I^c}, \tilde z_I)$ does not violate the assumption $n_k(z)\in [\f {n}{\beta K}, \f {\beta n}K]$. We just need to check an extreme case. When $n$ is sufficiently large, and recall that $\beta>1$, we consider the case $z_{I^c}$ are all 2. We still have
\begin{align*}
n_1(z)\ge \Big( 1-\f {\beta-1}{2\beta} \Big) \f nK -1=\f{(\beta+1)n}{2\beta K}-1\ge \f n{\beta K},
\end{align*}
and
\begin{align*}
n_2(z)\le \f nK+ \f{(\beta-1)n}{2\beta K} + K
\le \f{\beta n}K.
\end{align*}
Moreover, the maximum of $\max_{k\ne \ell} \eta(p_{k*}, p_{\ell*})$ achieves when $k=1$ and $\ell=3$. It only depends on $p, q$ and $n_3(z)$, so it is fixed on the parameter space $\mathcal S'$.
Hence, any choice $z_{I^c}$ will not violate the assumption $D^*= \cds{\text{Bin}(\lfloor n/K \rfloor, p^*)}{\text{Bin}(\lfloor n/K \rfloor, q)}$. Therefore, the possible label vectors in $\mathcal S'$ are exactly all the elements of $Z_I$. As a result, we have
\begin{align}\label{eq:minimax:to:bayes}
\inf_{\hat z\in Z_I} \sup_{\mathcal S'} \f 1n\E[H(\hat z,z)]
=\inf_{\hat z\in Z_I} \sup_{z\in Z_I} \f 1n\E[H(\hat z,z)]
\ge \inf_{\hat z\in Z_I} \f 1{n|Z_I|} \sum_{z\in Z_I}\sum_{i\in I^c} \P(\hat z_i\ne z_i)
\end{align}
where the last inequality is due to the fact that minimax risk is lower bounded by Bayes risk with equal prior. Assigning equal probability to each label vector in $Z_I$ is equivalent as letting $\P(z_i=1)=\P(z_i=2)=\f 12$ independently for all $i\in I^c$, so the Bayes estimator $\hat z=(\hat z_I, \hat z_{I^c})$ on the RHS of \eqref{eq:minimax:to:bayes} satisfies $\hat z_I = z_I$ and
\begin{align*}
\hat z_{I^c} &= \arg\max_{z_{I^c}} \P(z_{I^c}|A)
=\arg\max_{z_{I^c}} \P(z_{I^c}|A_{I^c\times[n]})\\
&=\arg\max_{z_{I^c}} \P(A_{I^c\times[n]}|z_{I^c})P(z_{I^c})
=\arg\max_{z_{I^c}} \prod_{i\in I_c}\P(A_{i*}|z_{I^c})
\end{align*}
Given $z\in Z_I$, then $z_I$ is fixed, the random vector $A_{i*}$ only depends on $z_i$. If $z_i=1$, then $\E[A_{i*}]=q\ones_n$. If $z_i=2$, then $\E[A_{i j}]=
\begin{cases}
p, \text{if }z_j=3;\\
q, \text{otherwise.}
\end{cases}$
Hence we have $\P(A_{i*}|z_{I^c})=\P(A_{i*}|z_i)$ for $i\in I^c$.
For every fixed $i\in I^c$ and $k=1$ or $2$, $|z\in Z_I: z_i=k| = |Z_I|/2$. Let $\psi_1\sim \text{Bin}(\lfloor n/K \rfloor, p)$ and $\psi_2\sim \text{Bin}(\lfloor n/K \rfloor, q)$
\begin{align*}
\f 1{n|Z_I|}\inf_{\hat z\in Z_I} \sum_{z\in Z_I} \sum_{i=1}^n\P(\hat z_i\ne z_i)
&= \f 1{n}\sum_{i\in I^c}  \inf_{\hat z_i\in \{1,2\}} \f {1}2 [\P(\hat z_i= 2|z_i=1) + \P(\hat z_i\ne 1|z_i=2)]\\
& = \f {|I^c|}{2n} \eta(\psi_1, \psi_2)
\ge \f {n\eta(\psi_1, \psi_2) }{2 n \big( 17\vee \f{2}{\beta-1} \big)\beta K}.
\end{align*}
Finally, since $\log \omega'\le\max_{j\in[n]} \Big| \log \f{p_{kj}(1-p_{\ell j})}{p_{\ell j}(1-p_{kj})}  \Big|\le \log \f{\omega}{1-\eps}$, we have $$\sqrt{np}\max_{j\in[n]}\Big|\log\f{p_{kj}(1-p_{\ell j})}{p_{\ell j}(1-p_{kj})} \Big| =\Theta (\sqrt{np}).$$ By Lemma \ref{lem:fund:limit}, we have $\eta(\psi_1, \psi_2) = \Omega\Big( \frac 1{\sqrt{np}} e^{-D^*} \Big)$.

\subsection{Auxiliary lemmas for Theorem~\ref{thm:alg:mis:upper:bound}}

In this section, we will use the following concentration inequality~\cite[p.~118]{gine2015mathematical}:

\begin{prop}[Prokhorov]\label{prop:prokh:concent}
	Let $S = \sum_{i} X_i$ for independent centered variables $\{X_i\}$, each bounded by $c < \infty$ in absolute value a.s. and suppose $v\ge \sum_i \E X_i^2$, then for $t>0$,
	\begin{align}\label{eq:poi:concent}
	\P\big(S > vt \big) \le \exp[ {- v h_c(t)} ], \quad \text{where}\;
	h_c(t) := \frac3{4c} t \log \big(1+\frac{2c}{3} t\big).
	\end{align}
	Same bound holds for $\P(S < -vt)$.
\end{prop}

\begin{lemma}[Uniform Parameter Estimation]\label{lem:Phat:bound}
	For $\hat P$ obtained from the operation $\mathcal B(A,\tilde z)$, and assuming $\mis(\tilde z, z)\le\gamma$ for $\f 1n\le \gamma \le \f 1{2\beta K}$ with optimal permutation $\pi^*=$id, we have
	\begin{align*}
	\P\Big( \sup\{	\|\hat P-P\|_\infty: \summ i n 1\{\tilde z_i\ne z_i\}\le n\gamma\} \le C(8\beta K \gamma+\tau)p^*\Big)
	\le \exp\Big[ -\f{n^2p^*h_1(\tau)}{8\beta^2K^2}  - n\gamma\log \gamma \Big].
	\end{align*}
	If $\gamma<\f 1n$, we can replace $n\gamma\log \gamma$ by 0.
\end{lemma}

\begin{proof}
	We only consider the case $k= \ell$. If $k\ne\ell$, the arguments will similarly follow. Let $E=\{(i,j): \tilde z_i = z_i = \tilde z_j=\tilde z_j = k\}$, $F = \{(i,j): \tilde z_i = \tilde z_j = k, \text{ but } z_i\ne k \text{ or } z_j\ne k\}$. Let $\hat n_k=|\{i\in[n]: \tilde z_i=k\}|$. According to assumptions,
	\begin{align*}
	(\hat n_k-n\gamma)^2\le |E|\le \hat n_k^2  \quad \text{and}\quad |F|\le 2n\gamma\hat n_k.
	\end{align*}
	Hence by definition of $\hat P_{k\ell}$ from \eqref{eq:operator:B}, we have upper bound
	\begin{align*}
	\hat n_k^2 \E[\hat P_{k\ell}]
	\le |E|P_{k\ell} +|F|p^*
	\le \hat n_k^2 P_{k\ell} + 2n\gamma\hat n_kp^*
	\end{align*}
	Since $\hat n_k\ge n_k-n\gamma\ge \f n{\beta K} - \f n{2\beta K}=\f n{2\beta K}$,  so
	\begin{align*}
	\E[\hat P_{k\ell}]\le \f{\hat n_k^2 P_{k\ell} +2n\gamma\hat n_kp^*}{\hat n_k ^2}\le P_{k\ell}+4\beta K\gamma p^*.
	\end{align*}
	For lower bound, we have
	\begin{align*}
	\hat n_k^2\E[\hat P_{k\ell}] &\ge (\hat n_k-n\gamma)(\hat n_k-n\gamma-1)P_{k\ell}\ge \hat n_k \hat n_\ell P_{k\ell}-2n\gamma(\hat n_k+1)p^*\\
	&\ge \hat n_k \hat n_\ell P_{k\ell}-4n\gamma\hat n_kp^*.
	\end{align*}
	Therefore,
	\begin{align*}
	\E[\hat P_{k\ell}]\ge \f{\hat n_k^2 P_{k\ell} -4n\gamma\hat n_kp^*}{\hat n_k^2}\ge P_{k\ell}-8\beta K\gamma p^*.
	\end{align*}
	Thus $|P_{k\ell}-\E[\hat P_{k\ell}]|\le 8\beta K\gamma p^*$. By Proposition \ref{prop:prokh:concent},
	\begin{align*}
	\P(|\hat P_{k\ell}-\E[\hat P_{k\ell}]|&\ge \tau p^*)
	=\P\Big(\f{\hat n_k(\hat n_k-1)}2(\hat P_{k\ell}-\E[\hat P_{k\ell}])\ge \f{\hat n_k(\hat n_k-1)}2\tau p^*\Big)\\
	&\le 2\exp\Big[-\f{\hat n_k(\hat n_k-1)}2 p^* h_1(\tau)\Big]
	\le 2\exp\Big[ -\f{n^2p^*h_1(\tau)}{8\beta^2K^2} \Big].
	\end{align*}
	There are at most
	\begin{align*}
	\summm i 0 {\lfloor n\gamma \rfloor} {n \choose i}
	\le {n \choose {\lfloor n\gamma \rfloor}}\summm j 0 \infty \Big(\f {\lfloor n\gamma \rfloor}{n-\lfloor n\gamma \rfloor+1}\Big)^j
	\le \Big( \f{en}{n\gamma} \Big)^{n\gamma}\Big( \f{n-2n\gamma+1}{n-n\gamma+1} \Big)
	\le \exp(-n\gamma\log\gamma).
	\end{align*}
	different $\tilde z$ with error rate at most $\gamma$. If $\gamma<\f 1n$, then $\tilde z=z$ is unique. Taking the union bound, we obtain the desired probability.
\end{proof}

Let $\varphi(x;p)$ be the PMF evaluated at $x$ of a Poisson-Binomial variable with parameters $p=(p_1,\dots, p_n)$. In particular, if $p = \bar p \ones_n$, then $\varphi(x,p)$ is the PMF of binomial distribution with parameters $n$ and $\bar p$.

\begin{lemma}[Binomial Perturbation]\label{lem:bin:pert}
	If $|\bar p_1 - \bar p_2|\le \delta \max(\bar p_1, \bar p_2):=\delta p^*$, $p^*/\omega\le \bar p_1, \bar p_2\le 1-\varepsilon$, then
	\begin{align*}
	\f{\varphi(x; \bar p_1 \ones_{n_1})}{\varphi(x; \bar p_2 \ones_{n_2})}\le \exp\Big( \delta \omega x + \f{\delta n_2 p^*}\varepsilon +(n_2-n_1)p^*\Big) \quad \forall x \in \ints_+.
	\end{align*}
\end{lemma}

\begin{proof}
	Using $1+x\le e^x$ several times, we have
	\begin{align*}
	\f{\varphi(x; \bar p_1 \ones_n)}{\varphi(x; \bar p_2 \ones_n)}
	&=\Big(\f{\bar p_1}{\bar p_2}\Big)^x \f{(1-\bar p_1)^{n_1-x}}{(1-\bar p_2)^{n_2-x}}
	=\Big( \f{\bar p_1}{\bar p_2} \Big)^x \Big(\f{1-\bar p_1}{1-\bar p_2}\Big)^{n_2-x}(1-\bar p_1)^{n_1-n_2}\\
	&\le \Big( \f{\bar p_2+\delta p^*}{\bar p_2} \Big)^x \Big(\f{1-\bar p_2+\delta p^*}{1-\bar p_2}\Big)^{n_2-x}\exp(\bar p_1|n_1-n_2|)\\
	&\le \Big( 1+ \delta \omega \Big)^x \Big( 1+\f{\delta p^*}{\varepsilon} \Big)^{n_2-x}\exp(\bar p_1(n_2-n_1))\\
	&\le \exp\Big( \delta \omega x + \f{\delta n_2 p^*}\varepsilon + (n_2-n_1)p^*\Big).
	\end{align*}
\end{proof}

\begin{lemma}[Poisson-Binomial Approximation]\label{lem:poi:bin:appr}
	Let $p=(p_1,\dots,p_n)$ be parameter of a Poisson Binomial distribution. Let $p^*:=\max_{i\in[n]}p_i$. We assume at least $n(1-\gamma)$ entries of $p$ are exactly $\bar p$, and $p^*\le \min(1-\eps, \omega\bar p)$. Then,
	\begin{align*}
	\f{\varphi(x;p)}{\varphi(x;\bar p \ones_n)}\le \exp \lp \f \gamma \eps (n\bar p + \omega x) \rp, \quad \forall x \in \ints_+.
	\end{align*}
\end{lemma}

\begin{proof}
	Let $\mathcal S(x) = \{S\in[n]: |S| = x\}$ for $x\in \ints_+$, then
	\begin{align*}
	\varphi(x;p) = \prod_{i=1}^n (1-p_i) \sum_{S\in\Sc(x)} \prod_{j\in S}\f {p_j}{1-p_j}
	\end{align*}
	By Maclaurin's inequality,
	\begin{align*}
	\sum_{S\in\Sc(x)}\prod_{j\in S}\f {p_j}{1-p_j} \le { n \choose x } \f 1 {n^x} \lp \sum_{i=1}^n \f{p_i}{1-p_i} \rp^x,
	\end{align*}
	so we have
	\begin{align*}
	\f{\varphi(x; p)}{\varphi(x; \bar p \ones_n)} &\le \f{\prod_{i=1}^n (1-p_i) \sum_{S\in\Sc(x)} \prod_{j\in S}\f {p_j}{1-p_j}}{{n\choose x}\bar p^x (1-\bar p)^x}
	\le \f{\prod_{i=1}^n (1-p_i) \f 1 {n^x} \lp \sum_{i=1}^n \f{p_i}{1-p_i} \rp^x}{\bar p^x (1-\bar p)^{n-x}}.
	\end{align*}
	Without loss of generality, we assume $p_{\lfloor n\gamma\rfloor +1}=\dots =p_n=\bar p$. We have
	\begin{align*}
	\f{\prod_{i=1}^n (1-p_i) \f 1 {n^x} \lp \sum_{i=1}^n \f{p_i}{1-p_i} \rp^x}{\bar p^x (1-\bar p)^{n-x}}
	&= \lp \prod_{i=1}^n \f{1-p_i}{1-\bar p}\rp \lp \f 1n \sum_{i=1}^n \f{p_i(1-\bar p)}{\bar p(1-p_i)} \rp^x\\
	&=\Big(\prod_{i=1}^{\lfloor n\gamma\rfloor} \f{1-p_i}{1-\bar p}\rp \lp 1-\gamma + \f 1n\sum_{i=1}^{\lfloor n\gamma\rfloor} \f{p_i(1-\bar p)}{\bar p(1-p_i)}\rp^x\\
	&\le \f 1{(1-\bar p)^{n\gamma}} \lp 1+ \f 1n\sum_{i=1}^{\lfloor n\gamma\rfloor} \f{p_i}{\bar p(1-p_i)}\rp^x.
	\end{align*}
	By the inequality $\f 1{1-x}\le \exp \big( \f x{1-x} \big)$ for $x\in(0,1)$, we have
	\begin{align*}
	\f 1{(1-\bar p)^{n\gamma}}\le \exp \lp \f{n\gamma\bar p}{1-\bar p} \rp \le \exp \lp \f{n\gamma\bar p}{\eps} \rp.
	\end{align*}
	For the other term, $1+x\le e^x$ implies
	\begin{align*}
	\lp 1+ \f 1n\sum_{i=1}^{\lfloor n\gamma\rfloor} \f{p_i(1-\bar p)}{\bar p(1-p_i)}\rp^x
	\le \exp \lp \f{x\gamma p^*}{\bar p (1-p^*)} \rp
	\le \exp \lp \f {\gamma \omega x}{\eps} \rp.
	\end{align*}
	Therefore, we have
	\begin{align*}
	\f{\varphi(x; p)}{\varphi(x; \bar p \ones_n)}
	\le \f{\prod_{i=1}^n (1-p_i) \f 1 {n^x} \lp \sum_{i=1}^n \f{p_i}{1-p_i} \rp^x}{\bar p^x (1-\bar p)^{n-x}}
	\le \exp \lp \f \gamma \eps (n\bar p + \omega x) \rp.
	\end{align*}
\end{proof}

\begin{lemma}[Degree Truncation]\label{lem:degree:truncation}
	For fixed $i\in [n]$, let $b_{i+}=\sum_{\ell\in[K]}b_{ir} = \sum_{j=1}^n A_{ij}$ be the degree of node~$i$, where $A$ is the adjacency matrix in SBM (see \eqref{eq:def:adjacency:matrix}), and assuming $\max_{j\in[n]} \E[A_{ij}] \le p^*\le 1-\eps$. Then there exists $C_\eps>0$, which only depends on $\eps$ such that
	\begin{align*}
	\P(b_{i+} > C_\eps n p^*) \le {(1-p^*)^n}\exp(-n p^*). 
	\end{align*}
\end{lemma}

\begin{proof}[Proof of Lemma~\ref{lem:degree:truncation}]
	There exists $C_\eps'$ such that for $x\in [0, 1-\eps]$, $e^{-C'_\eps x}\le 1-x$.
	We choose large enough $C_\eps$ such that
	\begin{align*}
	(C_\eps-1)\log \Big( 1+\f {2(C_\eps-1)}3 \Big) \ge C'_\eps+1.
	\end{align*}
	Now we want to find the upper bound of the following probability:
	\begin{align*}
	\P\big(\,b_{i+} > C_\eps np^* \big)\le \P\big(\,b_{i+}-\E[b_{i+}] > (C_\eps -1)np^*\big).
	\end{align*}
	For fixed $i$, let $p_j = A_{ij}$, and $v=\summ j n p_j(1-p_j)$, $vt=(C_\eps -1)n p^*$, so $t\ge C_\eps -1$. By Proposition~\ref{prop:prokh:concent}, we have
	\begin{align*}
	\P\big(\,b_{i+}-\E[b_{i+}] > (C_\eps-1) np^*\big)
	&\le \exp\Big[ {-}\f 34 vt \log\lp 1+\f{2t}3\rp \Big]\\
	&\le \exp\Big[ -(C_\eps-1)\log \Big( 1+\f {2(C_\eps-1)}3 \Big) np^* \Big]\\
	&\le \exp(-(C'_\eps+1) n p^*)\le (1-p^*)^n \exp(-n p^*).
	\end{align*}
	where the second to last inequality is by the assumption  $C'_\eps$ satisfying $e^{-C'_\eps x}\le 1-x$  for $x\in [0, 1-\eps]$, and the last inequality comes from the fact that $e^{-x}\le 1/\sqrt x$.
\end{proof}

\begin{lemma}[Bounds of $\alpha^*$]\label{lem:bounded:alpha}
	Under the setting in Section \ref{sec:example:bern},
	suppose
	\begin{align*}
	\max_{j\in[n]} \Big( \f{p_{0j}}{p_{1j}}\vee \f{p_{1j}}{p_{0j}} \Big)\le\omega,  \max_{j\in[n]}(p_{0j}\vee p_{1j})\le 1-\eps
	\text{ and }
	\alpha^* = \arg\max_{\alpha\in[0,1]}\cd{p_{0*}}{p_{1*}}
	\end{align*}
	for $\omega>1$ and $\eps\in(0,1)$, then there exists $\delta\in(0,1/2)$ which only depends on $\eps$ and $\omega$ such that $\alpha^*\in[\delta,1-\delta]$.
\end{lemma}

\begin{proof}
	We first consider the case $n=1$ and briefly denote $p_{01}:=p$ and $p_{11}:=q$ (in this proof only). Let $f(\alpha) = p^{1-\alpha}q^\alpha+(1-p)^{1-\alpha}(1-q)^\alpha$, then $$f'(\alpha) = p^{1-\alpha}q^\alpha \log \f qp+(1-p)^{1-\alpha}(1-q)^\alpha \log\f{1-q}{1-p}.$$
	Since $f$ is smooth and convex, $\alpha^*$ minimize $f(\alpha)$ if and only if $f'(\alpha^*)=0$. Let us define $x := \log \f qp$ and $y:=\log \f{1-q}{1-p}$, then $p = \f{1-e^y}{e^x-e^y}$ and $1-p = \f{e^x-1}{e^x-e^y}$. Without loss of generality, we assume $p<q$, so $x<0$ and $y>0$. Hence
	$$f'(\alpha) = \f{1-e^y}{e^x-e^y} xe^{\alpha x}+\f{e^x-1}{e^x-e^y} ye^{\alpha y}.$$
	$f'(\alpha^*)=0$ implies
	$$\alpha^* = \f{\log \f{e^y-1}y - \log \f{e^x-1}x}{y-x}.$$
	Let
	$$g(z) = \begin{cases}
	0, & \text{ if } z=0;\\
	\log\f{e^z-1}z &\text{otherwise.}
	\end{cases}$$
	We can observe that $g$ is a positive strictly increasing smooth function on $\reals$, and $g'\in(0,1)$. $\alpha^*$ is the slope of a secant line that intersects the function $g$ at $x$ and $y$, so $\alpha^*$ can only take value $g'(z)$ for some $z\in[x,y]$.  Since $x\in [-\log\omega, \log\omega]$ and $y\in [1-\eps, \f 1{1-\eps}]$, there exists $\delta$ which only depends on $\omega$ and $\eps$ such that $\alpha^*\in [\delta, 1-\delta]$.
	Now we can generalize the conclusion to $n>1$. Let
	$$f(\alpha) := \prod_{i=1}^n  f_j(\alpha) :=\prod_{i=1}^n [p_{0j}^{1-\alpha} p_{1j}^\alpha + (1-p_{0j})^{1-\alpha}(1-p_{1j})^\alpha] $$
	Since each positive convex function $f_j$ is decreasing on $[0,\delta]$ and increasing on $[1-\delta, 1]$, so is their product pointwise $f$. Therefore, $f$ achieves minimum on $[\delta, 1-\delta]$.
\end{proof}

\begin{lemma}\label{lem:likelihood:ratio}
	Let $b_{ir} := \sum_{j=1}^n A_{ij}1\{\hat z_j=r \}$,  $B(\rho) = \{\tilde P: \|P-\tilde P\|_\infty\le\rho\}$, $\hat n_r=|\{\hat z_j=r \}|$ and
	\begin{align*}
	Y_{ik\ell} &:= Y_{ik\ell}(\hat P, \hat z)\label{eq:def:yikl}\\
	&:= \sum_{j\ne i} A_{ij} \log \f{\hat P_{\ell \hat z_j}}{\hat P_{k\hat z_j}}+(1-A_{ij})\log \f{1-\hat P_{\ell \hat z_j}}{1-\hat P_{k \hat z_j}}
	=\sum_{r=1}^K b_{ir} \log \f{\hat P_{\ell r}}{\hat P_{kr}}+(\hat n_r-b_{ir})\log \f{1-\hat P_{\ell r}}{1-\hat P_{k r}}.
	\end{align*}
	Assuming $p^*:=\|P\|_\infty\le 1-\varepsilon$, $\rho\le \eps/2$, $\sum_{i=1}^n 1\{\hat z_i\ne z_i\}\le \gamma$ for $\f 1n\le \gamma \le \f 1{2\beta K}$ and $z_i=k$, then for some constant $C$ which only depends on $\omega$, $\beta$, $K^*$ and $\eps$, and for all $\ell\ne k$, we have
	\begin{align*}
	\P(\exists \hat P\in B(\rho), &Y_{ik\ell}(\hat P, \hat z)\ge 0) \\
	&\le \exp(C(n\rho + np^*\gamma )) \eta(p_{k*}, p_{\ell *})+{(1-p^*)^n}\exp(-n p^*).
	\end{align*}
	In particular, $$\P(\exists \hat P\in B(\rho), Y_{ik\ell}(\hat P, \hat z)\ge 0) \le \exp(C(n\rho + np^*\gamma )) \exp(-\cds{p_{k*}}{p_{\ell*}}).$$
\end{lemma}

\begin{proof}
	Firstly, we define the the following probability mass functions:
	\begin{align}
	\begin{split}
	\bar \psi_0\sim \bigotimes_{r=1}^K \text{Bin}&(\hat n_r, P_{kr}),
	\quad
	\hat \psi_0\sim \bigotimes_{r=1}^K \text{Bin}(\hat n_r, P_{kr}-\rho)\\
	\text{and}\quad
	&\hat \psi_1\sim \bigotimes_{r=1}^K \text{Bin}(\hat n_r, P_{\ell r}+\rho).
	\end{split}
	\end{align}
	Then we have
	\begin{align*}
	\sup_{\hat P\in B(\rho)}Y_{ik\ell}
	&=\sup_{\hat P\in B(\rho)} \sum_{r=1}^K b_{ir} \log \f{\hat P_{\ell r}}{\hat P_{kr}}+(\hat n_r-b_{ir})\log \f{1-\hat P_{\ell r}}{1-\hat P_{k r}}\\
	&\le \sum_{r=1}^K b_{ir} \log \f{P_{\ell r}+\rho}{P_{kr}-\rho}+(\hat n_r-b_{ir})\log \f{1-P_{\ell r}+\rho}{1-P_{k r}-\rho}\\
	& \le C_1 n\rho + \sum_{r=1}^K b_{ir} \log \f{P_{\ell r}+\rho}{P_{kr}-\rho}+(\hat n_r-b_{ir})\log \f{1-P_{\ell r}-\rho}{1-P_{k r}+\rho}\\
	&\le C_1 n\rho + \f{\hat \psi_1(b_{i*})}{\hat \psi_0(b_{i*})}.
	\end{align*}
	where $C_1$ only depends on $\eps$. We define subsets of $Z_+^K$:
	We also define $\tilde \psi_0$ as the probability mass function of $(b_{ir})\in \mathbb Z_+^K$. Note that $b_{ir}$ follows Poisson binomial distribution with at least $\hat n_r-n\gamma-1$ parameters equal to $P_{kr}$. Since $\gamma\le \f 1{2\beta K}$, so $\hat n_r\ge \f{n}{\beta K}-\f{n}{2\beta K}=\f{n}{2\beta K}$. The proportion of parameters different from $P_{kr}$ is at most
	\begin{align*}
	\gamma_r:=\f{n\gamma+1}{\hat n_r}\le \f{2n\gamma}{\hat n_r}\le 2n\gamma \cdot \f{2\beta K}{n}\le 4\beta K\gamma
	\end{align*}
	By Lemma \ref{lem:poi:bin:appr}, then
	\begin{align}\label{eq:real:vs:pert:bin}
	\frac {\tilde \psi_0(x)}{\bar \psi_0(x)}\le \prod_{r=1}^K \exp \lp \f {4\beta K\gamma} \eps (\hat n_r P_{kr} + \omega x_r) \rp
	\le \exp \lp \f {4\beta K\gamma} \eps (n p^* + \omega \summ r K x_r) \rp.
	\end{align}
	By Lemma \ref{lem:bin:pert}, let $\delta:=\rho/p^*$, then
	\begin{align}\label{eq:pert:bin:vs:ext:bin}
	\f{\bar \psi_0(x)}{\hat \psi_0(x)}\le \prod_{r=1}^K\exp\Big( \delta \omega x_r + \f{\delta \hat n_r p^*}\varepsilon \Big)
	=\exp\Big( \delta \omega \sum_{r=1}^K x_r + \f{\delta n p^*}\varepsilon \Big).
	\end{align}
	We define subsets of $Z_+^K$:
	\begin{align*}
	E = \{x\in \mathbb Z_+^K: \sum_{i=1}^K x_r \le C_\eps np^* \},
	\end{align*}
	where $C_\eps$ is defined in Lemma \ref{lem:degree:truncation}. Then for $x\in E$, $\sum_{r=1}^K x_r\le C_\eps np^*$. We combine \eqref{eq:real:vs:pert:bin} and \eqref{eq:pert:bin:vs:ext:bin}, and have
	\begin{align}
	\f{\tilde \psi_0(x)}{\hat \psi_0(x)}\le \exp\Big( {C_2(\delta+\gamma)}n p^* \Big), \quad \forall x\in E
	\end{align}
	where $C_2$ only depends on $\beta, K^*, \eps$, and $\omega$. Hence we have
	\begin{align}
	\begin{split}
	&\P\Big(\sup_{\hat P\in B(\rho)}Y_{ik\ell}\ge 0\Big)
	\le  \sum_{x\in \mathbb Z_+^K} \tilde \psi_0(x) 1\Big\{ \log \f{\hat \psi_1(x)}{\hat \psi_0(x)}\ge -C_1 n \rho \Big\} \\
	&= \sum_{x\in \mathbb Z_+^K} \tilde \psi_0(x) 1\Big\{  \f{e^{C_1 n \rho}\hat \psi_1(x)}{\hat \psi_0(x)}\ge 1 \Big\} \\
	&\le \sum_{x\in E} \exp\Big( {C_2(\delta+\gamma)}\omega n p^* \Big)\hat \psi_0(x) 1\Big\{  \f{e^{C_1 n \rho}\hat \psi_1(x)}{\hat \psi_0(x)}\ge 1 \Big\} + \sum_{x\notin E} \psi_0(x)\\
	&\le  \exp\Big( {C_2(\delta+\gamma)}n p^*+ C_1 n \delta p^*\Big) \sum_{x\in E}\min(\hat \psi_0(x), \hat \psi_1(x)) + \f{(1-p^*)^n}{\sqrt{np^*}}\\
	&\le \exp((C_1+C_2) (\delta+\gamma)np^*)\sum_{x\in E}\min(\hat \psi_0(x), \hat \psi_1(x)) + \f{(1-p^*)^n}{\sqrt{np^*}}.\label{eq:yikl:bound:pert:bin}
	\end{split}
	\end{align}
	Now we consider the perturbation of total variation affinity. Let $\psi_0\sim \bigotimes_{r=1}^K \text{Bin}(n_r, P_{k r})$ and $\psi_1\sim \bigotimes_{r=1}^K \text{Bin}(n_r, P_{\ell r})$.  Then by Lemma \ref{lem:bin:pert}, we have
	\begin{align*}
	\f{\hat \psi_0(x)}{\psi_0(x)} \le \prod_{r=1}^K \exp\Big( \delta \omega x_r + \f{\delta n_r p^*}\varepsilon + (n_r-\hat n_r)p^*\Big)
	=\exp\Big( \delta\omega\sum_{r=1}^K x_r +\f{\delta np^*}{\eps} \Big).
	\end{align*}
	For $x\in E$, we have
	\begin{align*}
	\f{\hat \psi_0(x)}{\psi_0(x)}
	\le \exp\Big( \delta\omega\sum_{r=1}^K x_r +\f{\delta np^*}{\eps} \Big)
	\le \exp ( np^*\delta(C_\eps\omega + 1/\eps) ).
	\end{align*}
	Same bound holds for $\hat \psi_1/\psi_1$ on $E$. Therefore, we have
	\begin{align}\label{eq:yikl:bound:bin}
	\sum_{x\in E}\min(\hat \psi_0(x), \hat \psi_1(x)) \le \exp ( np^*\delta(C_\eps\omega + 1/\eps) )\sum_{x\in E}\min(\psi_0(x), \psi_1(x)).
	\end{align}
	We combine \eqref{eq:yikl:bound:pert:bin} and \eqref{eq:yikl:bound:bin} and obtain the desired result.
\end{proof}

\begin{lemma}\label{lem:frac:lower:bound}
	For any $C_1$, there exists $C_2$ only depends on $\beta, \eps, K^*$ and $\omega$ such that if $np^*\le C_2 \cds{p_{k*}}{p_{\ell*}}^2$, then $\sqrt n \bar \sigma\alpha^*(1-\alpha^*)\ge C_1$.
\end{lemma}

\begin{proof}
	There exists $j\in[n]$ such that
	\begin{align*}
	-\log(p_{kj}^{1-\alpha} p_{\ell j}^{\alpha}+(1-p_{kj})^{1-\alpha}(1-p_{\ell j})^\alpha )\ge \f{\cd{p_{k*}}{p_{\ell*}}}n.
	\end{align*}
	For sufficiently large $C_3$, which only depends on $\eps$, we have
	\begin{align*}
	C_3((1-\alpha)p_{kj}+\alpha p_{\ell j}-p_{kj}^{1-\alpha}p_{\ell j}^\alpha)\ge
	-\log(p_{kj}^{1-\alpha} p_{\ell j}^{\alpha}+(1-p_{kj})^{1-\alpha}(1-p_{\ell j})^\alpha ),
	\end{align*}
	so $C_4n((1-\alpha)p_{kj}+\alpha p_{\ell j}-p_{kj}^{1-\alpha}p_{\ell j}^\alpha)\ge \cd{p_{k*}}{np_{\ell*}}$. Dividing both side by $C_4 n p_{kj}$, we have
	\begin{align*}
	1-\alpha + \alpha \f{p_{\ell j}}{p_{k j}} -\Big(  \f{p_{\ell j}}{p_{k j}} \Big)^\alpha \ge \f{\cd{p_{k*}}{p_{\ell*}}}{C_4 np_{kj}}\ge \f{\cd{p_{k*}}{p_{\ell*}}}{C_4 np^*}.
	\end{align*}
	Let us define $f(x):=1-\alpha+\alpha x-x^\alpha$, then for $x>1$, we have
	\begin{align*}
	f_\alpha(x)\le \f 12(1-\alpha)\alpha(x-1)^2\le \f 18(x-1)^2.
	\end{align*}
	Without loss of generality, we assume $p_{\ell j}> p_{kj}$, then we have
	\begin{align*}
	\f 18\Big( \f{p_{\ell j}}{p_{kj}}-1 \Big)^2\ge f\Big( \f{p_{\ell j}}{p_{kj}} \Big)
	\f{\cd{p_{k*}}{p_{\ell*}}}{C_4 np_{kj}}\ge \f{\cd{p_{k*}}{p_{\ell*}}}{C_4 np^*}.
	\end{align*}
	which implies $\log \f{p_{\ell j}}{p_{kj}}\ge \f 12\log \Big( 1+\f{8\cd{p_{k*}}{p_{\ell*}}}{C_4 np^*}  \Big)$. Suppose $z_j = r$, by assumption, $n_r\ge n/(\beta K)$, so
	\begin{align*}
	n\bar{\sigma}_n&\ge n_r\Big( \log\f{p_{\ell j}(1-p_{kj})}{p_{k j}(1-p_{\ell j})}  \Big)^2(p_{\alpha j}(1-p_{\alpha j}))
	\ge \f{np^*\eps}{\omega\beta K} \Big( \log \f{p_{\ell j}}{p_{k j}}  + \log \f{1-p_{kj}}{1-p_{\ell j}}  \Big)^2\\
	&= \f{np^*\eps}{\omega\beta K}\Big( \f 12\log \Big( 1+\f{8\cd{p_{k*}}{p_{\ell*}}}{C_3 np^*}  \Big) \Big)^2.
	\end{align*}
	Since $(\log(1+x))^2-x^2+x^3\ge 0$, we have
	\begin{align*}
	\f{np^*\eps}{\omega\beta K}\Big( \f 12\log \Big( 1+\f{8\cd{p_{k*}}{p_{\ell*}}}{C_3 np^*}  \Big) \Big)^2
	\ge \f{np^*\eps}{4\omega\beta K}\Big( \Big(\f{8\cd{p_{k*}}{p_{\ell*}}}{C_3 np^*} \Big)^2-\Big(\f{8\cd{p_{k*}}{p_{\ell*}}}{C_3 np^*} \Big)^3\Big).
	\end{align*}
	By Lemma \ref{lem:D:bounds:by:np}, we have
	\begin{align*}
	\Big(\f{8\cd{p_{k*}}{p_{\ell*}}}{C_3 np^*} \Big)^2-\Big(\f{8\cd{p_{k*}}{p_{\ell*}}}{C_3 np^*} \Big)^3
	\ge \Big( \f{8^2}{C_3^2} - \f{8^3C_\eps}{C_3^3} \Big)\Big(\f{\cd{p_{k*}}{p_{\ell*}}}{np^*} \Big)^2
	\ge \Big( \f{8^2}{C_3^2} - \f{8^3C_\eps}{C_3^3} \Big)\f {C_2}{np^*}.
	\end{align*}
	Choose $C_3$ sufficiently large so that $\Big( \f{8^2}{C_3^2} - \f{8^3C_\eps}{C_3^3} \Big)>0$, then choose $C_2$ sufficiently large so that $\f{C_2\eps}{4\omega \beta K^*}\Big( \f{8^2}{C_3^2} - \f{8^3C_\eps}{C_3^3} \Big)>C_1$, then we have the desired result.
\end{proof}

\begin{lemma}[Random Partitioning]\label{lem:rand:part}
	Let $I$ be a random subset of $[n]$ with $|I|=\lfloor n/2\rfloor$ in Algorithm \ref{alg:lrc} and $n_k^I=|\{ i\in I: z_i=k \}|$, then for sufficiently large $n$,
	\begin{align*}
	\max_{k\in[K]}\Big| n_{k}^I - \frac {n_k}2\Big| \le n\xi
	\end{align*}
	holds with probability at least $1- 2K \exp\big(-n \xi^2 /3\big)$.
\end{lemma}

\begin{proof}[Proof of Lemma~\ref{lem:rand:part}]
	We have $n_{k}^I\sim \text{Hypergeometric}(\lfloor n/2\rfloor, n_k, n)$. For any fixed $k \in [\Kr]$, the concentration of hypergeometric distribution~\cite{chvatal1979tail} gives $\Big| n_{k}^I - \frac {n_k}2\Big| \le n\xi$ with probability at least $1-2\exp(-\nr \xi^2/3)$ when $n$ is sufficiently large. Taking the union bound over all $k\in[K]$ gives the desired result.
\end{proof}

\begin{lemma}\label{lem:D:bounds:by:np}
	Recall the Chernoff information $\cd{p_{k*}}{p_{\ell*}}$ between Bernoulli distribution from \eqref{eq:chernoff:bern}, assuming $\max_{i} \max(p_{ki}, p_{\ell i})=p^*\le 1-\eps$, then $\cd{p_{k*}}{p_{\ell*}}\le C_\eps np^*$ where $C_\eps$ only depends on $\eps$.
\end{lemma}
\begin{proof}
	For $\alpha\in[0,1]$, we have
	\begin{align*}
	\cd{p_{k*}}{p_{\ell*}} &= \sum_{j=1}^n -\log [p_{kj}^{1-\alpha} p_{\ell j}^{\alpha}+(1-p_{kj})^{1-\alpha}(1-p_{\ell j})^{\alpha}]\\
	&\le \sum_{j=1}^n -\log [(1-p_{kj})^{1-\alpha}(1-p_{\ell j})^{\alpha}]\\
	&=-\sum_{j=1}^n (1-\alpha)\log(1-p_{kj})+\alpha\log(1-p_{\ell j}).
	\end{align*}
	For $k\in [K]$ and $j\in[n]$, we have $1-p_{kj}\le \eps$, so $-\log(1-p_{kj})\le C_\eps p_{kj}$ where $C_\eps$ only depends on $\eps$. Hence for every $\alpha\in[0,1]$,
	\begin{align*}
	\cd{p_k}{p_\ell} \le \sum_{j=1}^n (1-\alpha) Cp_{kj}+\alpha Cp_{\ell j}\le C_\eps np^*.
	\end{align*}
\end{proof}

\subsection{Proof of Theorem~\ref{thm:alg:mis:upper:bound}}

Under the assumption $C_1\le np^*\le C_2 (D^*)^2$, we can assume $C_2$ is sufficiently small, then after $C_2$ is fixed, we can assume $n$, $np^*$ and $D^*$ are sufficiently large by choosing big enough $C_1$. Now we will analyze the algorithm step by step. Each step fails with some probability, which will be summed up before calculating the error rate.\\

{\bf Spectral clusterings and matching.} Assuming $D^*:=\min_{k\ne\ell} \cds{p_{k}}{p_{\ell}}$ is sufficiently large, we have $\mis(\tilde z ,z)\le C_3((D^*)^{-1})\le\f{1}{8\beta K}$ with probability at least $1-n^{-(r+1)}$, because $\beta$ is fixed and $K=O(1)$. Without loss of generality, we assume the optimal permutation between $\tilde z$ and $z$ are identity, that is, $n\mis(\tilde z, z) = \sum_{i=1}^n 1\{\tilde z_i\ne z_i\}$. Now we consider spectral clustering in the for loop. Using Lemma \ref{lem:rand:part} and let $\xi = \f 1{6\beta K}$, when $n$ is sufficiently large, we have
\begin{align*}
\f{n_k}{3}\le \f{n_k}{2}-\f{n}{6\beta K}\le n_k^I:=|\{i\in I: z_i=k \}|.
\end{align*}
For sufficiently large $n$, we have $n/(4\beta K)\le n_k/4\le n_k^{I'}$. Similar bound holds for $n_k^{J'}$, i.e., $n_k/4\le n_k^{J'}$. Hence for $\alpha\in(0,1)$,
\begin{align}\label{eq:cd:I:lower}
\begin{split}
\cd{p_{kI'}}{p_{\ell I'}}&=-\sum_{j\in I'}\log(p_{ki}^{1-\alpha}p_{\ell i}^\alpha+p_{k i}^{1-\alpha}p_{\ell i}^\alpha)\\
&\ge -\sum_{r=1}^K \f{n_k}4\log(P_{kr}^{1-\alpha}P_{\ell r}^\alpha+P_{k r}^{1-\alpha}P_{\ell r}^\alpha)
=\frac{\cd{p_{k}}{p_{\ell}}}4\ge D^*/4,
\end{split}
\end{align}
Then the output $\tilde z'_{I'}$ of first spectral clustering in step 6 satisfies $\mis(\tilde z'_{I'}, z_{I'})\le C_3(D^*/4)^{-1}\le \frac 1{8\beta K}$ when $D^*$ is sufficiently large with probability at least $1-(n/2-1)^{-(r+1)}\ge 1-(n/3)^{-(r+1)}$.
Now we consider the first matching algorithm in step 7. Let
\begin{align*}
\pi^* = \arg\max_{\pi\in S^K} \sum_{i\in I'} 1\{\tilde z_i\ne \pi(\tilde z'_i)\},
\end{align*}
then
$\sum_{i\in I'} 1\{ \tilde z_i\ne \pi^*(\tilde z'_i) \}\le \frac {|I'|}{8\beta K}\le \frac{n}{16\beta K}.$  On the other hand, since $n_k^{I'}\ge \frac {n}{4\beta K}$, we must have $|\{i\in I': \tilde z'_i=k \}|\ge \frac n{4\beta K}-\frac {n}{16\beta K}=\frac {3n}{16\beta K}$. Hence for every $k\in[K]$, $|\{ i\in I': \tilde z_i=\pi^*(\tilde z'_i)=k \}|\ge \frac{3n}{16\beta K}-\frac{n}{16\beta K}=\frac{n}{8\beta K}$. On the other hand, for any $\pi\ne \pi^*$,  $\sum_{i\in I'} 1\{ \tilde z_i\ne \pi(\tilde z'_i) \}\ge 2\cdot \frac{n}{8\beta K}=\frac {n}{4\beta K}$ because at least two labels have been permuted and at least $\f n{4\beta K}$ of them match $\tilde z$ under the permutation $\pi^*$. Then by triangle inequality of the hamming distance, we have
\begin{align*}
\sum_{i\in I'}1\{z_i\ne \pi(\tilde z'_i)\} \ge \sum_{i\in I'}1\{\tilde z_i\ne \pi(\tilde z'_i)\}-\sum_{i\in I'}1\{z_i\ne \tilde z_i\}
\ge \frac {n}{4\beta K}-\f{n}{16\beta K}=\f{3n}{16\beta K}.
\end{align*}
Therefore, $\pi^*$ is the unique permutation such that $\sum_{i\in I'}1\{ z_i\ne \pi^*(\tilde z'_i) \}\le \frac n{16\beta K}$. In other words, the matching algorithm succeed to find the optimal permutation between $\tilde z'_{I'}$ and $z_{I'}$. The second matching algorithm will similarly work. Therefore, the updated $\tilde z_{I'}$ and $\tilde z_{J'}$ are consistent with $z$.\\

{\bf First estimated parameters.} We will apply Lemma \ref{lem:Phat:bound} to find the bound for $\tilde P$ in step 2. By assumption, we have $\gamma\le C_3(D^*)^{-1}$ and we let $\tau_1=C_4D^*/(np^*)$ (where $C_4$ will be chosen later), then using $h_1(\tau)\ge \tau^2/8$, we have
\begin{align*}
\frac{n^2p^* h_1(\tau_1)}{4\beta^2K^2}&\ge \f{n^2p^*}{32\beta^2K^2}\Big( \frac{C_4D^*}{np^*} \Big)^2 = \f{C_4^2 n (D^*)^2}{32\beta^2K^2 np^*}\ge \f{C_4^2 n }{32C_2\beta^2K^2 }\\
&\ge  \f{2n\log (D^*/C_3)}{D^*/C_3}\ge -2n \gamma_1 \log \gamma_1
\end{align*}
where the second to the last inequality holds when $D^*$ is sufficiently large, and the last inequality is due to the fact that $-x\log x$ is increasing on $[0,1/e]$. Therefore, with failing probability at most
\begin{align}\label{eq:fail:pr:1}
\begin{split}
\exp\Big[ -\f{n^2p^*h_1(\tau_1)}{4\beta^2K^2}  &- n\gamma\log \gamma \Big]
\le \exp\Big[ -\f{n^2p^*h_1(\tau_1)}{8\beta^2K^2} \Big]\\
&\le \exp\Big[ -\f{C_2^2 n (D^*)^2}{32\beta^2K^2 np^*} \Big]
\le \exp(-2D^*),
\end{split}
\end{align}
we have $\|\tilde P-P\|_\infty \le C_5(8\beta K\gamma_1+\tau_1)p^*$ fails, where $C_5$ corresponds to constants in Lemma \ref{lem:Phat:bound}, and the last inequality holds for sufficiently large $D^*$.\\

{\bf First likelihood ratio test.}
In step 8, we apply likelihood ration test on $A_{I'\times J'}$. We recall the definition of $Y_{ik\ell}$ in \eqref{eq:def:yikl} from Lemma \ref{lem:likelihood:ratio}. The updated $\tilde z'_{I'}$ satisfies $\tilde z'_i=z_i$ if $Y_{iz_i\ell}<0$ for every $\ell\ne z_i$. For $\tilde P\in B(\rho):=\{ \hat P: \|\hat P-P\|_\infty\le\rho \}$,  the probability that the classification error rate on nodes $I'$ is at least $\gamma_2$ after the first likelihood ratio test is
\begin{align}
\begin{split}
&\P\Big( \sum_{i\in I'} 1\{\max_{\ell\ne z_i}Y_{iz_i\ell}(\tilde P, \tilde z'_{I'})\ge 0\} \ge |I'|\gamma_2\Big) \\
&\le \P\Big( \sum_{i\in I'} 1\{\exists \tilde P\in B(\rho), \max_{\ell\ne z_i} Y_{iz_i\ell}(\tilde P, \tilde z'_{I'})\ge 0\} \ge |I'|\gamma_2\Big).\label{eq:dep:to:ind}
\end{split}
\end{align}
Let us define random variable
\begin{align*}
Z_i=1\{\exists \tilde P\in B(\rho), \max_{\ell\ne z_i} Y_{iz_i\ell}(\tilde P, \tilde z'_{I'})\ge 0\}, \quad\text{for} \quad i\in I'.
\end{align*}
Since $\tilde z'_{I'}$ only depends on $A_{I'\times J'}$, which is independent with $A_{I'\times J'}$. $\sum_{i\in I'}Z_i$ is a summand of independent variables. We can assume $\tilde z'_{I'}$ is fixed and satisfying $\mis(\tilde z'_{I'} \z_{I'})\le 4C_3((D^*)^{-1}):=\gamma_1$. We apply Lemma \ref{lem:likelihood:ratio} with $\rho:=\rho_1:=C_5(8\beta K\gamma_1+\tau_1)p^*$, and let $C_6$ be the constant in the lemma,
\begin{align*}
\P(Z_i=1)
&\le K\exp (C_6n(\rho_1+p^*\gamma_1)) \min_{\ell\ne z_i}\exp(-\cds{p_{z_iJ'}}{ p_{\ell J'}})\\
&\le K\exp (C_6np^*( C_5(8\beta K\gamma_1+\tau_1) + \gamma_1 ))\exp(-D^*/4)\\
&\le K\exp(4C_3C_6(8C_5\beta K+1)np^*/D^* + C_6 np^*(C_4D^*/(np^*)))\exp(-D^*/4)\\
&\le K\exp(4C_3C_6(8C_5\beta K+1)np^*/D^* + C_6 C_4D^*)\exp(-D^*/4).
\end{align*}
By choosing $C_2$ sufficiently small, and using $np^*\le C_2((D^*)^2)$, we have
\begin{align*}
4C_3C_6(8C_5\beta K+1)np^*/D^* \le 4C_3C_6(8C_5\beta K+1)C_2 D^*\le D^*/24.
\end{align*}
Then we choose $C_4=1/(6C_6)$, we have $C_6C_4 D^*\le D^*/24$. When $D^*$ is sufficiently large, we have $\log K\le D^*/24$ because $K=O(1)$. Therefore, we have
\begin{align*}
\P(Z_i=1)
\le \exp(-D^*/4+D^*/24+D^*/24+D^*/24)
=  \exp(-D^*/8).
\end{align*}
Applying Proposition \ref{prop:prokh:concent} to\eqref{eq:dep:to:ind}, let $v=|I'|\exp(-D^*/8)$, $vt:=|I'|\gamma_2:=3|I'|e^{-D^*/16}+ 64$, then $t=e^{D^*/16}+\f{64}{|I'|} e^{D^*/8}$, so the failing probability
\begin{align}
\begin{split}
&\P\Big( \sum_{i\in I'} Z_i \ge 3|I'|e^{-D^*/16}+64\Big)\\
&=\P\Big( \sum_{i\in I'} Z_i-|I'|\P(Z_i=1) \ge 3|I'|e^{-D^*/16}+64-|I'|\P(Z_i=1)\Big)\\
&\le\P\Big( \sum_{i\in I'} Z_i-|I'|\P(Z_i=1) \ge 2|I'|e^{-D^*/16}+64\Big)\\
&\le \exp\Big(-(2|I'|e^{-D^*/16}+ 64)\frac 3{4}\log\Big(1+\frac {4e^{D^*/16}}3\Big)\Big)\\
&\le \exp(-2D^*)\label{eq:fail:pr:2}
\end{split}
\end{align}
The same error rate holds for $\tilde z'_{J'}$. Therefore, the updated $\tilde z'$ satisfies $\mis(\tilde z' ,z)\le 3e^{-D^*/16}+129/n$.\\

{\bf Second estimated parameters.} As we have obtained labels $\tilde z'$ with higher accuracy, we would like to update $\tilde P$ as well. The proof is similar as the first estimated parameter, but with $\tau_2$ and $\gamma_2$ different from $\tau_1$ and $\gamma_1$. Let $\tau_2:=\frac {16\beta K(1\vee \sqrt{p^*D^*})}{np^*}$, using $h_1(\tau)\ge \tau^2/8$ for $\tau_2$ again, we have
\begin{align*}
\f{n^2p^*h_1(\tau_2)}{4\beta^2K^2}\ge  \frac{256n^2 p^* (1\vee p^*D^*)}{32(np^*)^2} = 8D^*\vee \frac{8}{p^*}.
\end{align*}
Let $\gamma_2:=3e^{-D^*/16}+129/n\le 6e^{-D^*/16}\vee 258/n$. Since $-x\log x$ is increasing on $[0,1/e]$, we have
\begin{align*}
-n\gamma_2\log \gamma_2
&\le (-n(6e^{-D^*/16})\log(6e^{-D^*/16}))\vee(-258\log\Big( \f{258}n \Big))\\
&=6ne^{-D^*/16}\Big(\f{D^*}{16}-\log 6\Big)\vee (258\log n-258\log 258)\\
&\le ne^{-D^*/8}\vee (258\log n).
\end{align*}
Again, we want to show that $\f{n^2p^*h_1(\tau_2)}{4\beta^2K^2}\ge -2n\gamma_2\log \gamma_2$. Using $np^*\le C_2 (D^*)^2$, we have
\begin{align*}
\frac 8{p^*}=\frac {8n}{np^*}\ge 8ne^{-\f{\sqrt{np^*}}{8\sqrt{C_2}}}\ge 8ne^{-D^*/8}\ge 2ne^{-D^*/8}.
\end{align*}
when $np^*$ is sufficiently large. By $np^*\le C_2 (D^*)^2$ again, we have $\f n{C_2}\le \f{(D^*)^2}{p^*}$, so and either $D^*$ or $1/p^*$ is greather than $\sqrt{\f n{C_2}}$, and greater than $516 \log n$ when $n$ is sufficiently large. Hence
\begin{align*}
\f{n^2p^*h_1(\tau_2)}{4\beta^2K^2}
\ge  8D^*\vee \frac{8}{p^*}
\ge 2ne^{-D^*/8}\vee (516\log n)
\ge -2n\gamma_2\log\gamma_2.
\end{align*}
Therefore, by Lemma \ref{lem:Phat:bound}, with failing probability at most,
\begin{align}\label{eq:fail:pr:3}
\exp\Big[ -\f{n^2p^*h_1(\tau_2)}{4\beta^2K^2}  - n\gamma_2\log \gamma_2 \Big]
\le \exp\Big( -4D^*\vee \frac{4}{p^*} \Big) \le \exp(-2D^*).
\end{align}
we have $\|\hat P-P\|_\infty \le C_5(8\beta K\gamma_2+\tau_2)p^*$.\\

{\bf Second Likelihood ratio test.} The arguments will be similar as the first likelihood ratio test.  We define define $Z_j$ by new $\rho$, $\hat P$ and $\tilde z'$, i.e.,
\begin{align*}
Z_j=1\{\exists \tilde P\in B(\rho), \max_{\ell\ne z_i} Y_{iz_i\ell}(\tilde P, \tilde z')\ge 0\}.
\end{align*}
The likelihood ratio test $\hat z_j \gets \mathcal L(A_{j*},\hat P, \tilde z')$ in step 10 succeed to recover $z_j$ if $Z_j=0$. We apply Lemma \ref{lem:likelihood:ratio} with $\rho:=\rho_2:=C_5(8\beta K\gamma_2+\tau_2)p^*$, and let $\eta^*=\max_{k\ne\ell} \eta(p_{k*}, p_{\ell*})$, then we have
\begin{align}
\begin{split}
\P(Z_j=1)
\le& K\exp (C_6n(\rho_2+p^*\gamma_2)) \eta^*+K(1-p^*)^ne^{-np^*}+\\
&3\exp(-2D^*)+n^{-(r+1)}+2\Big( \f n3 \Big)^{-(r+1)}\\
\le& K\exp (C_6(8C_5\beta K+1)np^*\gamma_2 + C_5C_6np^*\tau_2)\eta^*\\
&+K(1-p^*)^ne^{-np^*}+3\exp(-2D^*) + 3\Big( \f{3}n \Big)^{r+1} \label{eq:p:zj}
\end{split}
\end{align}
where $3\exp(-2D^*)$ comes from the failing probability of first parameter estimation \eqref{eq:fail:pr:1}, first likelihood ratio test \eqref{eq:fail:pr:2}, and second likelihood ratio test \eqref{eq:fail:pr:3}, $n^{-(r+1)}$ is the failing probability of spectral clustering in step 1, and $2\big( \f n3 \big)^{-(r+1)}$ is the failing probabilities in step 7. We recall that $\gamma_2:=3e^{-D^*/16}+129/n$, then by assumption $np^*\le C_2 (D^*)^2$, we have $3e^{-D^*/16}np^*=O(1)$ when $D^*$ is sufficiently large. To handle the last term, we apply Lemma \ref{lem:fund:limit}, then we have
\begin{align*}
\eta(p_{k*}, p_{\ell*})&\ge C_7 \Big( \sqrt{np^*}\max_{j\in[n]}\Big|\log\f{p_{kj}(1-p_{\ell j})}{p_{\ell j}(1-p_{kj})} \Big| \Big)^{-1}\exp(-\cds{p_{k*}}{p_{\ell*}})
\end{align*}
where $C_7$ is the constant of the lower bound in the lemma. Then by \eqref{eq:abs:constants},
\begin{align*}
\Big|\log\f{p_{kj}(1-p_{\ell j})}{p_{\ell j}(1-p_{kj})}\Big|\le \log \f{\omega}{1-\eps},
\end{align*}
and by \eqref{eq:chernoff:bern}, we have
\begin{align*}
\exp(-\cds{p_{k*}}{p_{\ell*}}) = \prod_{i=1}^n p_{kj}^{1-\alpha}p_{\ell j}^\alpha + (1-p_{kj})^{1-\alpha}(1-p_{\ell j})^\alpha
\ge (1-p^*)^n.
\end{align*}
Hence, when $np^*$ is sufficiently large, we have $(1-p^*)^ne^{-np^*}\le \eta(p_{k*}, p_{\ell*})$ for all $k\ne \ell$. Thus, $(1-p^*)^ne^{-np^*}\le \eta^*$. Similarly, we also have
\begin{align*}
C_7 \Big( \sqrt{np^*}\max_{j\in[n]}\Big|\log\f{p_{kj}(1-p_{\ell j})}{p_{\ell j}(1-p_{kj})} \Big| \Big)^{-1} \ge \exp(-D^*)
\end{align*}
when $D^*$ is sufficiently large, so $3\eta^*\ge 3\exp(-2D^*)$. Combining these results, we have
\begin{align}\label{eq:p:zj:bound}
\P(Z_j=1)\le C_8\exp(C_5C_6np^*\tau_2)\eta^*+C_9\eta^*+ \f{3^{r+2}}{n^{r+1}}.
\end{align}
for some constants $C_8$ and $C_9$. We recall $\tau_2:=\frac {16\beta K(1\vee \sqrt{p^*D^*})}{np^*}$. To analyze the upper bound of $np^*\tau_2 = 16\beta K(1\vee \sqrt{p^*D^*})$, we will consider two cases: \\
Case 1. $D^*\le (1+r/2)\log n$. $np^*\le C_2(D^*)^2$ implies $np^*\le \f{C_2  ((r+2)\log n)^2}4$, so
\begin{align*}
16\beta K\sqrt {p^*D^*} = 16\beta K\sqrt {\f{np^*D^*}n}
\le 16\beta K\sqrt {\f{C_2(r+2)^3(\log n)^3}{8n}} = O(1).
\end{align*}
when $n$ is sufficiently large. Hence $np^*\tau_2=O(1)$. On the other hand, for all $k\ne \ell$.
\begin{align*}
\eta(p_{k*}, p_{\ell*})&\ge C_7 \Big( \sqrt{np^*}\max_{j\in[n]}\Big|\log\f{p_{kj}(1-p_{\ell j})}{p_{\ell j}(1-p_{kj})} \Big| \Big)^{-1}\exp(-\cds{p_{k*}}{p_{\ell*}})\\
&=\Omega (np^*)^{-1/2}\exp(-\cds{p_{k*}}{p_{\ell*}})\\
&=\Omega \Big( \f{4}{(r+2)^2(\log n)^2}e^{-\big( 1+\f r2 \big)\log n} \Big)
=\Omega(n^{-{r+1}}).
\end{align*}
Applying this result to \eqref{eq:p:zj:bound}, we have the output $\hat z_j$ satisfies $\P(\hat z_j\ne z_j)=\P(Z_j=1)=O(\eta^*)$. Since this is true for all $j\in[n]$, so $\E[\mis(\hat z, z)]\le \f 1n\sum_{j=1}^n \P(\hat z_j\ne z_j)=O(\eta^*)$.\\
Case 2. $D^*\ge (1+2s)\log n$. In this case
\begin{align*}
P(Z_i=1)&\le C_{10}\exp(C_5C_6 np^*\tau_2)\eta^*+\f{3^{r+2}}{n^{r+1}}\\
&\le C_{10}\exp(-D^*+16C_5C_6\beta K(1\vee \sqrt{p^*D^*}))+\f{3^{r+2}}{n^{r+1}}\\
&\le \exp\big(-D^*+sD^*\big)+\f{3^{r+2}}{n^{r+1}}\\
&\le \exp(-(1-s)(1+2s)\log n)+\f{3^{r+2}}{n^{r+1}}
\le \f 1{n^{1+s}}+\f{3^{r+2}}{n^{r+1}}.
\end{align*}
Therefore, $\E[\mis(\hat z,z)]\ =O\Big( \f 1{n^{(1+r)\wedge (1+s)}} \Big)$.
Since $D^*\ge (1+s)\log n$ where $s>0$, then we have $\eta^*\le e^{-D^*}\le 1/n$. $\mis(\hat z,z)\ge \eta^*$ means $\hat z$ fails to recover at least one node, which is equivalent to $\mis(\hat z,z)\ge 1/n$. By Markov inequality, for positive random variable $X$ and $t>0$, we have $\P(X\ge t)\le \E[X]/t$. Let $t=1/n$, we have
\begin{align*}
\P\Big(\mis(\hat z,z)\ge \eta^*\Big)
&=\P\Big(\mis(\hat z,z)\ge \frac 1n\Big)=O\Big(\frac n{n^{(1+r)\wedge (1+s)}}\Big)=O\Big( \f 1{n^{r\wedge s}} \Big).
\end{align*}
Finally, $\eta*$ can be replaced by equivalent expression in Lemma \ref{lem:fund:limit}.

\printbibliography
\end{document}

%% file: bi_pl_defs.tex
\newcommand{\z}{z} 

\newcommand{\nr}{n}

\newcommand{\Kr}{K}

\newcommand{\eps}{\varepsilon}
\newcommand{\ints}{\mathbb Z}

\usepackage{mathrsfs}

\newcommand{\f}{\frac}
\newcommand{\E}{\mathbb E}

\newcommand{\summ}[2]{\sum_{#1 = 1}^{#2}}
\newcommand{\summm}[3]{\sum_{#1 = #2}^{#3}}

\renewcommand{\P}{\mathbb P}

\DeclareMathOperator{\mis}{Mis}

\newcommand{\lp}{\Big(}
\newcommand{\rp}{\Big)}

\newcommand{\Are}{A_{\text{re}}}

\DeclareMathOperator{\SC}{SC}

\newcommand{\beq}{\begin{eqnarray*}}
	\newcommand{\eeq}{\end{eqnarray*}}
\newcommand{\bal}{\begin{align}}
\newcommand{\eal}{\end{align}}

\newcommand{\Var}{\text{Var}}

\newcommand{\vertiii}[1]{{\left\vert\kern-0.25ex\left\vert\kern-0.25ex\left\vert #1 
		\right\vert\kern-0.25ex\right\vert\kern-0.25ex\right\vert}}

\newcommand{\match}[0]{{\textsc{Match}}\xspace}

\newcommand{\lr}{{l}}

\newcommand{\Sc}{\mathcal{S}}

\newcommand{\ones}{{\bf 1}}

\renewcommand{\SC}{\text{SC}}

\newtheorem{theorem}{Theorem}[section]
\newtheorem{corollary}{Corollary}
\newtheorem{lemma}{Lemma}

\newtheorem{prop}{Proposition}
\newtheorem{remark}{Remark}

\newcommand{\cd}[2]{D_\alpha({#1}\|{#2})}
\newcommand{\cds}[2]{D_{\alpha^*}({#1}\|{#2})}
\newcommand{\kld}[2]{D_{\text{KL}}({#1}\|{#2})}
\newcommand{\tv}[2]{D_{\text{TV}}({#1}\|{#2})}